\documentclass[smallextended,numbook]{svjour3}     
\smartqed  
\usepackage{graphicx}
\usepackage{amsfonts}
\usepackage{amssymb}
\usepackage{amsmath}
\usepackage{amsxtra}
\usepackage{subfigure}
\usepackage{epstopdf}
\usepackage{latexsym}
\begin{document}

\title{Nonstationary iterated Tikhonov regularization in Banach spaces with
uniformly convex penalty terms
}

\titlerunning{Nonstationary iterated Tikhonov regularization}        

\author{Qinian Jin   \and Min Zhong     
}


\institute{Qinian Jin \at
Mathematical Sciences Institute,
Australian National University, Canberra, ACT 0200, Australia\\
\email{Qinian.Jin@anu.edu.au}\\
\\
Min Zhong \at
School of Mathematical Sciences, Fudan University, Shanghai 200433, China\\
\email{09110180007@fudan.edu.cn}\\
}


\newtheorem{Assumption}{Assumption}[section]
\newtheorem{Rule}{Rule}[section]
\newtheorem{Test}{Test}[section]

\def\l{\langle}
\def\r{\rangle}
\def\a{\alpha}
\def\b{\beta}
\def\d{\delta}
\def\la{\lambda}
\def\p{\partial}
\def\vep{{\mathcal E}}
\def\R{{\mathcal R}}

\def\N{\mathcal N}
\def\R{\mathcal R}
\def\D{\mathcal D}
\def\X{\mathcal X}
\def\Y{\mathcal Y}
\def\B{\mathcal B}
\def\A{\mathcal A}

\maketitle

\begin{abstract}
We consider the nonstationary iterated Tikhonov regularization in Banach spaces which defines
the iterates via minimization problems with uniformly convex penalty term. The penalty term is allowed to be non-smooth
to include $L^1$ and total variation (TV) like penalty functionals, which are significant in reconstructing
special features of solutions such as sparsity and discontinuities in practical applications.
We present the detailed convergence analysis and obtain the regularization property when the method is
terminated by the discrepancy principle. In particular we establish the strong convergence and the convergence in
Bregman distance which sharply contrast with the known results that only provide weak convergence
for a subsequence of the iterative solutions. Some numerical experiments on linear integral equations of first kind
and parameter identification in differential equations are reported.
\end{abstract}

\def\theequation{\thesection.\arabic{equation}}
\catcode`@=12

\section{\bf Introduction}
\setcounter{equation}{0}

We are interested in solving inverse problems which can be formulated as the operator equation
\begin{equation}\label{1.1}
F(x)=y,
\end{equation}
where $F: D(F)\subset \X\mapsto \Y$ is an operator between two Banach
spaces $\X$ and $\Y$ with domain $D(F)\subset \X$; the norms in $\X$ and $\Y$ are denoted by the
same notation $\|\cdot\|$ that should be clear from the context. A characteristic property of
inverse problems is their ill-posedness in the sense that their solutions do not depend continuously
on the data. Due to errors in the measurements, one never has the exact data in practical applications;
instead only noisy data are available. If one uses the algorithms developed for well-posed
problems directly, it usually fails to produce any useful information since noise could be amplified by
an arbitrarily large factor. Let $y^\d$ be the only available noisy data to $y$ satisfying
\begin{equation}\label{1.3}
\|y^\d-y\|\le \d
\end{equation}
with a given small noise level $\delta> 0$. How to use $y^\d$ to produce a stable approximate solution
to (\ref{1.1}) is a central topic, and regularization methods should be
taken into account.

When both $\X$ and $\Y$ are Hilbert spaces, a lot of regularization methods have been proposed
to solve inverse problems in the Hilbert space framework (\cite{EHN96,KNS2008}). In case $F:\X\to \Y$ is a bounded linear operator,
nonstationary iterated Tikhonov regularization is an attractive iterative method in which a sequence
$\{x_n^\d\}$ of regularized solutions is defined successively by
$$
x_n^\d :=\arg \min_{x\in \X} \left\{\frac{1}{2} \|F x-y^\d\|^2 + \frac{\a_n}{2} \|x-x_{n-1}^\d\|^2 \right\},
$$
where $x_0^\d:=x_0\in \X$ is an initial guess and $\{\a_n\}$ is a preassigned sequence of positive numbers.
Since $\{x_n^\d\}$ can be written explicitly as
$$
x_n^\d=x_{n-1}^\d -(\a_n I +F^* F)^{-1} F^* (F x_{n-1}^\d-y^\d),
$$
where $F^*: \Y\to \X$ denotes the adjoint of $F:\X\to \Y$, the complete analysis of the regularization property
has been established (see \cite{HG98} and references therein) when $\{\a_n\}$ satisfies suitable property
and the discrepancy principle is used to terminate the iteration,  This method has been extended in
\cite{Jin2010,Jin2011} to solve nonlinear inverse problems in Hilbert spaces.

Regularization methods in Hilbert spaces can produce good results when the sought solution is smooth.
However, because such methods have a tendency to over-smooth solutions, they may not produce good
results in applications where the sought solution has special features such as sparsity
or discontinuities. In order to capture the special features, the methods in Hilbert spaces should be
modified by incorporating the information of suitable adapted penalty functionals, for which the theories
in Hilbert space setting are no longer applicable.

The nonstationary iterated Tikhonov regularization has been extended in \cite{JS2012} for solving linear
inverse problems in Banach spaces setting by defining $x_n^\d$ as the minimizer of the convex minimization problem
$$
\min_{x\in \X} \left\{\frac{1}{r}\|F x-y^\d\|^r +\a_n \Delta_p(x, x_{n-1}^\d) \right\}
$$
for $n\ge 1$ successively, where $1\le r<\infty$, $1<p<\infty$ and $\Delta_p(\cdot, \cdot)$ denotes
the Bregman distance on $\X$ induced by the convex function $x\to \|x\|^p/p$. When $\X$ is uniformly smooth
and uniformly convex, and when the method is terminated by the discrepancy principle,
the regularization property has been established if $\{\a_n\}$ satisfies $\sum_{n=1}^\infty \a_n^{-1}=\infty$.
The numerical simulations in \cite{JS2012} indicate that
the method is efficient in sparsity reconstruction when choosing $\X=L^p$ with $p>1$ close to $1$ on one hand,
and provides robust estimator in the presence of outliers in the noisy data when choosing $\Y=L^1$ on the
other hand. However, since $\X$ is required to be uniformly smooth and uniformly convex and since
$\Delta_p(\cdot, \cdot)$ is induced by the power of the norm in $\X$, the result in \cite{JS2012}
does not apply to regularization methods with $L^1$ and total variation like penalty terms that are
important for reconstructing sparsity and discontinuities of sought solutions.

The total variational regularization was introduced in \cite{ROF92}, its importance was recognized immediately
and many successive works were conducted in the last two decades. In \cite{OBGXY2005} an iterative regularization
method based on Bregman distance and total variation was introduced to enhance the multi-scale nature of reconstruction.
The method solves (\ref{1.1}) with $F:\X\to \Y$ linear and $\Y$ a Hilbert space by defining
$\{x_n^\d\}$ in the primal space $\X$ and $\{\xi_n^\d\}$ in the dual space $\X^*$ via
\begin{align}\label{OBGXY}
\begin{split}
x_n^\d & :=\arg \min_{x\in \X} \left\{ \|F x-y^\d\|^2 +\a_n D_{\xi_{n-1}^\d} \Theta(x, x_{n-1}^\d)\right\},\\
\xi_n^\d & := \xi_{n-1}^\d- \frac{1}{\a_n} F^* (F x_n^\d-y^\d),
\end{split}
\end{align}
where $\Theta: \X \to (-\infty, \infty]$ is a proper convex function, $x_0^\d\in \X$ is an initial guess,
$\xi_0^\d\in \X^*$ is in the sub-gradient of $\Theta$ at $x_0^\d$, and $D_\xi\Theta(\cdot, \cdot)$ denotes
the Bregman distance induced by $\Theta$. This method was extended in \cite{BB2009} to solve nonlinear
inverse problems. Extensive numerical simulations were reported
in \cite{OBGXY2005,BB2009} and convergence analysis was given, with special attention to the case that $\X=L^2(\Omega)$
and $\Theta(x)=a \|x\|_{L^2}^2 +\int_\Omega |D x|$, where $\int_\Omega |D x|$ denotes the total variation,
when the iteration is terminated by a discrepancy principle and $\{\a_n\}$ satisfies the
condition $\underline{\a}\le \a_n \le \overline{\a}$ for two positive
constants $\overline{\a}\ge \underline{\a}>0$.  The analysis in \cite{OBGXY2005,BB2009}, however, is somewhat
preliminary since it provides only the boundedness of $\{\Theta(x_{n_\d}^\d)\}$ which guarantees only weak
convergence for a subsequence of $\{x_{n_\d}^\d\}$, where $n_\d$ denotes the stopping index determined by the
discrepancy principle. It is natural to ask if the whole sequence converges strongly and in Bregman distance.

We point out that the method (\ref{OBGXY}) is equivalent to the augmented Lagrangian method introduced originally
in \cite{H1969,P1969} and developed further in various directions, see \cite{IK2008} and reference therein.
One may refer to \cite{FG2012} for some results on convergence and convergence rates of the augmented Lagrangian
method applied to linear inverse problems in Hilbert spaces with general convex penalty term.
When $\X$ and $\Y$ are Hilbert spaces and $\Theta(x)=\|x\|^2$, (\ref{OBGXY}) is exactly the nonstationary
iterated Tikhonov regularization. In this paper we formulate an extension of the nonstationary iterated
Tikhonov regularization in the spirit of (\ref{OBGXY}) to solve (\ref{1.1}) with both $\X$ and $\Y$ being Banach
spaces and present the detailed convergence analysis when the method is terminated by the discrepancy principle.
In the method we allow $\{\a_n\}$ to vary in various ways so that geometric decreasing sequence can be included;
this makes it possible to terminate the method in fewer iterations. Moreover, we allow the penalty term
$\Theta$ to be general uniformly convex functions on $\X$ so that the method can be used for sparsity reconstruction
and discontinuity detection. Most importantly, we obtain
\begin{align*}
x_{n_\d}^\d\rightarrow x^\dag, \quad  \Theta(x_{n_\d}^\d) \rightarrow \Theta(x^\dag) \quad
\mbox{and} \quad D_{\xi_{n_{\delta}}^{\delta}} \Theta(x^\dag, x_{n_{\delta}}^{\delta}) \rightarrow 0
\end{align*}
and give a characterization of the limit $x^\dag$, which significantly improve the known convergence results.

This paper is organized as follows. In section 2 we give some preliminary results
on Banach spaces and convex analysis. In section 3, we then formulate the method in Banach spaces with uniformly
convex penalty term for solving linear and nonlinear inverse problems, and present the main convergence results.
In section 4 we first prove a convergence result for the method when the data is given exactly;
we then show that, if the data contains noise, the method
is well-defined and admits some stability property; by combining these results we finally obtain the proof
of the main convergence theorems.  Finally, in section 5 we present some numerical simulations on linear
integral equations of first kind and parameter identification problems in partial differential equations to
test the performance of the method.

\section{\bf Preliminaries}\label{Sect2}
\setcounter{equation}{0}

Let $\X$ be a Banach space with norm $\|\cdot\|$. We use $\X^*$ to denote its dual space. Given $x\in \X$
and $\xi \in \X^*$ we write $\l \xi, x\r=\xi (x)$ for the duality pairing. We use ``$\rightarrow$"
and ``$\rightharpoonup$" to denote the strong convergence and weak convergence respectively. If $\Y$ is
another Banach space and $A: \X\to \Y$ is a bounded linear operator, we use $A^*: \Y^*\to \X^*$ to denote its adjoint, i.e.
$\l A^* \zeta, x\r=\l \zeta, A x\r$ for any $x\in \X$ and $\zeta\in \Y^*$.  We use $\N(A)=\{x\in \X: A x=0\}$ to denote the
null space of $A$ and define
$$
\N(A)^\perp:= \{ \xi\in \X^*: \l \xi, x\r=0 \mbox{ for all } x\in \N(A)\}.
$$
When $\X$ is reflexive, there holds
\begin{equation}\label{eq:28.1june}
\N(A)^\perp =\overline{\R(A^*)},
\end{equation}
where $\R(A^*)$ denotes the range space of $A^*$ and $\overline{\R(A^*)}$ denotes the closure of $\R(A^*)$ in $\X^*$.

For a convex function $\Theta: \X \to (-\infty, \infty]$, we use $D(\Theta):=\{x\in \X: \Theta(x)<+\infty\}$
to denote its effective domain. We call $\Theta$ proper if $D(\Theta)\ne \emptyset$. Given $x\in \X$ we define
$$
\p \Theta(x):=\{\xi\in \X^*: \Theta(\bar x)-\Theta(x)-\l \xi, \bar x-x\r \ge 0 \mbox{ for all } \bar x\in \X\}.
$$
Any element $\xi\in \p\Theta(x)$ is called a subgradient of $\Theta$ at $x$. The multi-valued mapping $\p \Theta: \X\to 2^{\X^*}$ is called
the subdifferential of $\Theta$. It could happen that $\p \Theta(x)=\emptyset$ for some $x\in D(\Theta)$. Let
$$
D(\p\Theta):=\{x\in D(\Theta): \p \Theta(x)\ne \emptyset\}.
$$
For $x \in D(\p \Theta)$ and $\xi\in \p \Theta(x)$
we define
$$
D_\xi \Theta(\bar x,x):=\Theta(\bar x)-\Theta(x)-\l \xi, \bar x-x\r, \qquad \forall \bar x\in \X
$$
which is called the Bregman distance induced by $\Theta$ at $x$ in the direction $\xi$. Clearly
$D_\xi \Theta(\bar x,x)\ge 0$. By straightforward calculation one can see that
\begin{equation}\label{4.3.1}
D_\xi \Theta(x_2,x)-D_\xi \Theta(x_1, x) =D_{\xi_1} \Theta(x_2,x_1) +\l \xi_1-\xi, x_2-x_1\r
\end{equation}
for all $x, x_1\in D(\p \Theta)$, $\xi\in \p \Theta(x)$, $\xi_1\in \p \Theta(x_1)$ and $x_2\in \X$.

A proper convex function $\Theta: \X \to (-\infty, \infty]$ is called uniformly convex if there is a continuous function
$h:[0, \infty) \to [0, \infty)$, with the property that $h(t)=0$ implies $t=0$, such that
\begin{equation}\label{eq:7.29.1}
\Theta(\la \bar x +(1-\la) x) +\la (1-\la) h(\|\bar x-x\|) \le \la \Theta(\bar x) +(1-\la) \Theta(x)
\end{equation}
for all $\bar x, x\in \X$ and $\la\in (0,1)$. If $h$ in (\ref{eq:7.29.1}) can be taken as $h(t)=c t^p$
for some $c>0$ and $p\ge 2$, then $\Theta$ is called $p$-uniformly convex.
It can be shown (\cite[Theorem 3.5.10]{Z2002}) that $\Theta$ is uniformly convex if and only if there is a strictly increasing continuous function
$\varphi: [0, \infty)\to [0, \infty)$ with $\varphi(0)=0$ such that
\begin{equation}\label{pconv}
D_\xi \Theta(\bar x,x) \ge \varphi(\|\bar x-x\|)
\end{equation}
for all $\bar x \in \X$, $x\in D(\p \Theta)$ and $\xi\in \p \Theta(x)$.


On a Banach space $\X$, we consider for $1<r<\infty$ the convex function $x\to \|x\|^r/r$.
Its subdifferential at $x$ is given by
$$
J_r(x):=\left\{\xi\in \X^*: \|\xi\|=\|x\|^{r-1} \mbox{ and } \l \xi, x\r=\|x\|^r\right\}
$$
which gives the duality mapping $J_r: \X \to 2^{\X^*}$ with gauge function $t\to t^{r-1}$.
We call $\X$ uniformly convex if its modulus of convexity
$$
\d_{\X}(t) := \inf \{2 -\|\bar x +x\| : \|\bar x\|=\|x\|=1,  \|\bar x-x\| \ge t\}
$$
satisfies $\d_{\X}(t)>0$ for all $0 < t\le 2$. If there are $c>0$ and $r>1$ such that $\d_{\X}(t) \ge c t^r$
for all $0<t\le 2$, then $\X$ is called $r$-uniformly convex. We call $\X$ uniformly smooth if its modulus of
smoothness
$$
\rho_{\X}(s) := \sup\{\|\bar x+ x\|+\|\bar x-x\|- 2 : \|\bar x\| = 1,  \|x\|\le s\}
$$
satisfies $\lim_{s\searrow 0} \frac{\rho_{\X}(s)}{s} =0$.
One can refer to \cite{Adams,C1990} for many examples of Banach spaces, including the sequence spaces $l^r$,
the Lebesgue spaces $L^r$, the Sobolev spaces $W^{k,r}$ and the Besov spaces $B^{s,r}$
with $1<r<\infty$, that are both uniformly convex and uniformly smooth.

It is well known that any uniformly convex or uniformly smooth Banach space is reflexive.
On a uniformly smooth Banach space $\X$, every duality mapping $J_r$ with $1<r<\infty$ is single valued and
uniformly continuous on bounded sets; for each $1<r<\infty$ we use
\begin{align*}
\Delta_r(\bar{x},x) = \frac{1}{r}\|\bar{x}\|^r - \frac{1}{r}\|x\|^r - \l J_r(x),\bar{x}-x\r,\quad\forall \bar{x}, x\in\X
\end{align*}
to denote the Bregman distance induced by the convex function $\Theta(x) = \|x\|^r/r$.

Furthermore, on a uniformly convex Banach space, any sequence
$\{x_n\}$ satisfying $x_n \rightharpoonup x$ and $\|x_n\|\rightarrow \|x\|$ must satisfy $x_n\rightarrow x$
as $n\rightarrow \infty$. This property can be easily generalized for uniformly convex functions which we state
in the following result.

\begin{lemma}\label{lem:Kadec}
Let $\Theta:\X \to (-\infty, \infty]$ be a proper, weakly lower semi-continuous, and uniformly convex function.
Then $\Theta$ admits the Kadec property, i.e. for any sequence $\{x_n\}\subset \X$ satisfying
$x_n \rightharpoonup x\in \X$ and $\Theta(x_n) \rightarrow \Theta(x)<\infty$ there holds
$x_n\rightarrow x$ as $n\rightarrow \infty$.
\end{lemma}

\begin{proof}
Assume the result is not true. Then, by taking a subsequence if necessary, there is an $\epsilon>0$ such that
$\|x_n-x\| \ge \epsilon$ for all $n$. In view of the uniformly convexity of $\Theta$, there is a $\gamma>0$
such that
$
\Theta\left((x_n+x)/2\right) \le \left(\Theta(x_n)+\Theta(x)\right)/2 -\gamma.
$
Using $\Theta(x_n)\rightarrow \Theta(x)$ we then obtain
$$
\limsup_{n\rightarrow \infty} \Theta\left(\frac{x_n+x}{2}\right) \le \Theta(x) -\gamma.
$$
On the other hand, observing that $(x_n+x)/2 \rightharpoonup x$, we have from the weakly lower semi-continuity
of $\Theta$ that
$$
\Theta(x) \le \liminf_{n\rightarrow \infty} \Theta\left(\frac{x_n+x}{2}\right).
$$
Therefore $\Theta(x) \le \Theta(x)-\gamma$, which is a contradiction. \hfill $\Box$
\end{proof}

In many practical applications, proper, weakly lower semi-continuous, uniformly convex functions
can be easily constructed. For instance, consider $\X=L^p(\Omega)$, where $2\le p<\infty$ and
$\Omega$ is a bounded domain in ${\mathbb R}^d$. It is known that the functional
$
\Theta_0(x) := \int_\Omega |x(\omega)|^p d\omega
$
is uniformly convex on $L^p(\Omega)$ (it is in fact $p$-uniformly convex). Consequently we obtain on
$L^p(\Omega)$ the uniformly convex functions
\begin{equation}\label{eq:7.27}
\Theta(x):=\mu \int_\Omega |x(\omega)|^p d\omega + a \int_\Omega |x(\omega)| d\omega
+b \int_\Omega |D x|,
\end{equation}
where $\mu>0$, $a, b\ge 0$, and $\int_\Omega|D x|$ denotes the total variation of $x$ over $\Omega$ that is defined by (\cite{G84})
$$
\int_\Omega |D x| :=\sup\left\{ \int_\Omega x \, \mbox{div} f d\omega:
f\in C_0^1(\Omega; {\mathbb R}^N) \mbox{ and }
\|f\|_{L^\infty(\Omega)}\le 1\right\}.
$$
For $a=1$ and $b=0$ the corresponding function is useful for sparsity reconstruction (\cite{T96});
while for $a=0$ and $b=1$ the corresponding function is useful for detecting the discontinuities,
in particular, when the solutions are piecewise-constant (\cite{ROF92}).

\section{The method and main results}\label{Se2}

We now return to (\ref{1.1}), where $F:\X \to \Y$ is an operator between two Banach spaces $\X$ and $\Y$.
We will always assume that $\X$ is reflexive, $\Y$ is uniformly smooth, and (\ref{1.1}) has a solution.
In general, the equation (\ref{1.1}) may have many solutions.
In order to find the desired one, some selection criteria should be enforced. Choosing a proper convex
function $\Theta$, we pick $x_0\in D(\partial \Theta)$ and $\xi_0\in \partial \Theta(x_0)$ as the
initial guess, which may incorporate some available information on the sought solution. We define $x^{\dag}$
to be the solution of (\ref{1.1}) with the property
\begin{align}\label{eq definition of xdag}
D_{\xi_0} \Theta(x^{\dag},x_0) := \min_{x\in D(\Theta)\cap D(F) }\left\{D_{\xi_0} \Theta(x,x_0) : F(x) = y\right\}.
\end{align}

We will work under the following conditions on the convex function $\Theta$ and the operator $F$.

\begin{Assumption}\label{A0}
$\Theta$ is a proper, weakly lower semi-continuous and uniformly convex function such that (\ref{pconv}) holds, i.e.
there is a strictly increasing continuous function $\varphi: [0, \infty)\to [0, \infty)$ with $\varphi(0)=0$ such that
$$
D_\xi \Theta(\bar x, x) \ge \varphi(\|\bar x-x\|)
$$
for $\bar x \in \X$, $x\in D(\p \Theta)$ and $\xi\in \p \Theta(x)$.

\end{Assumption}

\begin{Assumption}\label{A1}
\begin{enumerate}
\item[]
\begin{itemize}
\item[(a)] $D(F)$ is convex, and $F$ is weakly closed, i.e. for any sequence $\{x_n\} \subset D(F)$
satisfying $x_n\rightharpoonup x\in \X$ and  $F(x_n)\rightharpoonup v \in \Y$ there hold $x\in D(F)$ and $F(x) = v$;

\item[(b)] There is $\rho>0$ such that (\ref{1.1}) has a solution in $B_\rho(x_0)\cap D(F)\cap D(\Theta)$, where $ B_{\rho}(x_0)
  : = \left\{x\in\X:\ \|x-x_0\|\leq \rho\right\}$;

\item[(c)]  $F$ is Fr\'{e}chet differentiable on $D(F)$,
and $x\to F'(x)$ is continuous on $D(F)$, where $F'(x)$ denotes the Fr\'{e}chet derivative of $F$ at $x$;

\item[(d)] There exists $0\le \eta<1$ such that
  \begin{align*}
  \|F(\bar{x})-F(x)-F'(x)(\bar{x}-x)\|\leq\eta\|F(\bar{x})-F(x)\|
  \end{align*}
  for all $\bar{x}, x\in B_{3\rho}(x_0)\cap D(F)$.

\end{itemize}
\end{enumerate}
\end{Assumption}

When $\X$ is a reflexive Banach space, by using the weakly closedness of $F$ and the weakly
lower semi-continuity and uniformly convexity of $\Theta$ it is standard to show that $x^\dag$ exists. The following result shows
that $x^\dag$ is in fact uniquely defined.

\begin{lemma}\label{lem existence and uniqueness of xdag}
Let $\X$ be reflexive,  $\Theta$ satisfy Assumption \ref{A0}, and $F$ satisfy Assumption \ref{A1}.
If $x^\dag$ is a solution of $F(x)=y$ satisfying (\ref{eq definition of xdag}) with
\begin{align}\label{eq:20.10june}
D_{\xi_0} \Theta(x^\dag, x_0) \le \varphi(\rho),
\end{align}
then $x^\dag$ is uniquely defined.
\end{lemma}

\begin{proof}
Assume that (\ref{1.1}) has two distinct solutions $\hat{x}$ and $x^\dag$ satisfying
(\ref{eq definition of xdag}). Then it follows from (\ref{eq:20.10june}) that
$$
D_{\xi_0} \Theta(\hat{x}, x_0) =D_{\xi_0}\Theta(x^\dag, x_0) \le \varphi (\rho).
$$
By using Assumption \ref{A0} on $\Theta$ we obtain $\|\hat{x}-x_0\| \le \rho$ and $\|x^\dag-x_0\|\le \rho$.
Since $F(\hat{x})=F(x^\dag)$, we can use Assumption \ref{A1} (d) to derive that $F'(x^\dag) (\hat{x}-x^\dag)=0$.
Let $x_\la=\la \hat{x}+(1-\la) x^\dag$ for $0<\la<1$. Then $x_\la\in B_{\rho}(x_0) \cap D(\Theta) \cap D(F)$
and $F'(x^\dag) (x_\la-x^\dag)=0$. Thus we can use Assumption \ref{A1} (d) to conclude that
$$
\|F(x_\la)-F(x^\dag)\| \le \eta \|F(x_\la)-F(x^\dag)\|.
$$
Since $0\le \eta <1$, this implies that $F(x_\la)=F(x^\dag)=y$. Consequently, by the minimal property of $x^\dag$ we have
\begin{align}\label{eq:20.7june}
D_{\xi_0} \Theta(x_\la, x_0) \ge  D_{\xi_0} \Theta(x^\dag, x_0).
\end{align}
On the other hand, it follows from the strictly convexity of $\Theta$ that
\begin{align*}
D_{\xi_0} \Theta (x_\la, x_0)
& < \la D_{\xi_0} \Theta (\hat{x}, x_0) +(1-\la) D_{\xi_0} \Theta (x^\dag, x_0)= D_{\xi_0} \Theta(x^{\dag}, x_0)
\end{align*}
for $0<\la<1$ which is a contradiction to (\ref{eq:20.7june}). \hfill $\Box$
\end{proof}

We are now ready to formulate the nonstationary iterated Tikhonov regularization with penalty term
induced by the uniformly convex function $\Theta$. For the initial guess
$x_0^{\delta}: = x_0\in D(\partial \Theta)\cap D(F)$ and $\xi_0^{\delta} := \xi_0\in \partial \Theta(x_0)$,
we take a sequence of positive numbers $\{\alpha_n\}$ and define the iterative sequences
$\{x_n^{\delta}\}$ and $\{\xi_n^{\delta}\}$ successively by
\begin{align}\label{eq method}
\begin{split}
&x_n^\d\in \arg \min_{x\in D(F)} \left\{\frac{1}{r} \|F(x)-y^\d\|^r
+ \a_n D_{\xi_{n-1}^\d} \Theta(x, x_{n-1}^\d)\right\},\\
&\xi_n^\d = \xi_{n-1}^\d-\frac{1}{\a_n} F'(x_n^\d)^* J_r(F(x_n^\d)-y^\d)
\end{split}
\end{align}
for $n\geq 1$, where $1<r <\infty$ and $J_r: \Y\to \Y^*$ denotes the duality mapping of $\Y$ with gauge function
$t\to t^{r-1}$ which is single-valued and continuous because $\Y$ is assumed to be uniformly smooth.
At each step, the existence of $x_n^{\delta}$ is
guaranteed by the reflexivity of $\X$ and $\Y$, the weakly lower semi-continuity and uniformly convexity of $\Theta$,
and the weakly closedness of $F$. However, $x_n^\d$ might not be unique when $F$ is nonlinear;
we will take $x_n^\d$ to be any one of the minimizers. In view of the minimality of $x_n^\d$, we have
$\xi_n^\d\in \p \Theta(x_n^\d)$. From the definition of $x_n^{\delta}$, it is straightforward to see that
\begin{align}\label{mono}
\|F(x_n^\d)-y^\d\|\leq \|F(x_{n-1}^\d)-y^\d\|,\qquad n=1,2,\cdots.
\end{align}
We will terminate the iteration by the discrepancy principle
\begin{align}\label{DP}
\|F(x_{n_\d}^\d)-y^\d\|\leq \tau\d<\|F(x_n^\d)-y^\d\|,\qquad 0\leq n<n_\d
\end{align}
with a given constant $\tau>1$. The output $x_{n_{\delta}}^{\delta}$ will be used to
approximate a solution of (\ref{1.1}).

In order to understand the convergence property of $x_{n_\d}^\d$, it is necessary to consider
the noise-free iterative sequences $\{x_n\}$ and $\{\xi_n\}$, where each $x_n$ and $\xi_n$ with $n\geq 1$
are defined by (\ref{eq method}) with $y^\d$ replaced by $y$, i.e.,
\begin{align}\label{eq:method}
\begin{split}
x_n&\in\arg\min_{x\in D(F)} \left\{\frac{1}{r} \|F(x)-y\|^r +\a_n D_{\xi_{n-1}}\Theta(x,x_{n-1})\right\}, \\
\xi_n&=\xi_{n-1}-\frac{1}{\a_n} F'(x_n)^* J_r(F(x_n)-y) \in \p \Theta(x_n).
\end{split}
\end{align}
In section \ref{noise-free} we will give a detailed convergence analysis on $\{x_n\}$; in particular, we
will show that $\{x_n\}$ strongly converges to a solution of (\ref{1.1}). In order to connect such result with the
convergence property of $x_{n_\d}^\d$, we will make the following assumption.

\begin{Assumption}\label{A2}
$x_n$ is uniquely defined for each $n$.
\end{Assumption}

We will give some sufficient condition for the validity of Assumption \ref{A2}. This assumption
enables us to establish some stability results connecting $x_n^\d$ and $x_n$
so that we can finally obtain the convergence property of $x_{n_\d}^\d$ in the following result.

\begin{theorem}\label{th2}
Let $\X$ be reflexive and $\Y$ be uniformly smooth, let $\Theta$ satisfy Assumption \ref{A0}, and let $F$ satisfy
Assumptions \ref{A1} and \ref{A2}. Assume that $1<r<\infty$, $\tau>(1+\eta)/(1-\eta)$ and
that $\{\alpha_n\}$ is a sequence of positive numbers satisfying
$\sum_{n=1}^{\infty}\alpha_n^{-1} = \infty$ and $\alpha_n\leq c_0\alpha_{n+1}$ for all $n$ with some
constant $c_0>0$. Assume further that
\begin{equation}\label{eq:20.8june}
D_{\xi_0} \Theta(x^\dag, x_0) \le \frac{\tau^r-1}{\tau^r-1+c_0} \varphi (\rho).
\end{equation}
Then, the discrepancy principle~\eqref{DP} terminates the method (\ref{eq method}) after $n_{\delta}<\infty$ steps.
Moreover, there is a solution $x_*\in D(\Theta)$ of (\ref{1.1}) such that
\begin{align} \label{eq:convergence}
x_{n_\d}^\d\rightarrow x_*, \quad  \Theta(x_{n_\d}^\d) \rightarrow \Theta(x_*) \quad
\mbox{and} \quad D_{\xi_{n_{\delta}}^{\delta}} \Theta(x_*, x_{n_{\delta}}^{\delta}) \rightarrow 0
\end{align}
as $\d\rightarrow 0$. If, in addition, $\N(F'(x^\dag))\subset \N (F'(x))$ for all $x\in B_{3\rho}(x_0)\cap D(F)$, then
$x_*=x^\dag$.
\end{theorem}

In this result, the closeness condition (\ref{eq:20.8june}) is used to
guarantee that $x_n^\d$ is in $B_{3\rho}(x_0)$ for $0\le n\le n_\d$ so that Assumption \ref{A1} (d) can
be applied. This issue does not appear when $F: \X\to \Y$ is a bounded linear operator. Furthermore,
Assumption \ref{A2} holds automatically for linear problems when $\Theta$ is strictly convex.
Consequently, we have the following convergence result for linear inverse problems.

\begin{theorem}\label{th3}
Let $F: \X \to \Y$ be a bounded linear operator with $\X$ being reflexive and $\Y$ being uniformly smooth,
let $\Theta$ be proper, weakly lower semi-continuous, and uniformly convex,
let $1<r<\infty$, and let $\{\alpha_n\}$ be such that
$\sum_{n=1}^{\infty}\alpha_n^{-1} = \infty$ and $\alpha_n\leq c_0\alpha_{n+1}$ for all $n$ with
$c_0>0$. Then, the discrepancy principle \eqref{DP} with $\tau>1$ terminates the method
after $n_{\delta}<\infty$ steps. Moreover, there hold
\begin{align*}
x_{n_\d}^\d\rightarrow x^\dag, \quad  \Theta(x_{n_\d}^\d) \rightarrow \Theta(x^\dag) \quad
\mbox{and} \quad D_{\xi_{n_{\delta}}^{\delta}} \Theta(x^\dag, x_{n_{\delta}}^{\delta}) \rightarrow 0
\end{align*}
as $\d\rightarrow 0$.
\end{theorem}

In the next section, we will give the detailed proof of Theorem \ref{th2}. It should be pointed out that
the convergence $x_{n_\d}^\d\rightarrow x_*$ does not imply $\Theta(x_{n_\d}^\d) \rightarrow \Theta(x_*)$ directly
since $\Theta$ is not necessarily continuous. The proof of $\Theta(x_{n_\d}^\d)\rightarrow \Theta(x_*)$
relies on additional observation.

When applying our convergence result to the situation that $\X=L^2(\Omega)$ and $\Theta(x)=\mu\int_\Omega |x(\omega)|^2 d\omega
+\int_\Omega |Dx|$ with $\mu>0$, we can obtain
$$
\|x_{n_\d}^\d-x^\dag\|_{L^2(\Omega)} \rightarrow 0 \quad \mbox{and} \quad
\int_\Omega |D x_{n_\d}^\d| \rightarrow \int_\Omega |Dx| \quad \mbox{ as } \d\rightarrow 0.
$$
This significantly improves the result in \cite{BB2009} in which only the boundedness of $\Theta(x_{n_\d}^\d)$
was derived and hence only weak convergence for a subsequence of $\{x_{n_\d}^\d\}$ can be guaranteed.

We conclude this section with some sufficient condition to guarantee the validity of Assumption \ref{A2}.

\begin{Assumption}\label{A3}
There exist $C_0 \ge 0$ and $1/r \le \kappa <1$ such that
\begin{align*}
 \|F(\bar x)-F(x) - F'(x)(\bar x-x)\| \le C_0 \left[D_{\xi} \Theta(\bar x,x)\right]^{1-\kappa}
\left[\Delta_r(F(\bar x)-y, F(x)-y)\right]^\kappa
\end{align*}
for all $\bar x, x\in B_{3\rho}(x_0)\cap D(\Theta) \cap D(F)$ with $x\in D(\p \Theta)$
and $\xi\in\partial \Theta(x)$, where $\Delta_r(\cdot, \cdot)$ denotes the Bregman distance on $\Y$ induced by
the convex function $\|y\|^r/r$.
\end{Assumption}

When $\Y$ is a $r$-uniformly convex Banach space, $\Theta$ is a $p$-uniformly convex function on $\X$ with $p\ge 2$,
and $1/p+1/r\le 1$, Assumption \ref{A3} holds with $\kappa=1-1/p$ if there is a constant $C_1\ge 0$ such that
\begin{equation}\label{A1-s}
\|F(\bar x)-F(x)-F'(x) (\bar x-x) \| \le C_1 \|\bar x-x\| \|F(\bar x) -F(x)\|
\end{equation}
for $\bar x, x\in B_{3\rho}(x_0)\cap D(F)$, which is a slightly strengthened version of Assumption~\ref{A1}~(d).

\begin{lemma}\label{lem:A3}
Let $\X$ be reflexive and $\Y$ be uniformly smooth, let $1<r<\infty$, let $\Theta$ satisfy Assumption \ref{A0},
let $F$ satisfy Assumptions \ref{A1} and \ref{A3}, and let $\{\a_n\}$ satisfy $\sum_{n=1}^\infty \a_n^{-1}=\infty$.
Assume that
\begin{align}\label{21.11june}
D_{\xi_0} \Theta(x^\dag, x_0)\le \varphi (\rho) \quad \mbox{ and } \quad
\bar{C}_0 \left[D_{\xi_0} \Theta(x^\dag, x_0)\right]^{1-\frac{1}{r}} <1
\end{align}
with $\bar{C}_0:=C_0 \kappa^\kappa (1-\kappa)^{1-\kappa} (1-\eta)^{\frac{1-r}{r}}
\a_1^{\kappa-\frac{1}{r}}$. Then Assumption \ref{A2} holds, i.e. $x_n$ is uniquely defined for each $n$.
\end{lemma}

We will prove Lemma \ref{lem:A3} at the end of Section \ref{subsect4.1} by using some useful estimates
that will be derived during the proof of the convergence of $\{x_n\}$.

\section{\bf Convergence analysis}\label{Sect3}
\setcounter{equation}{0}

We prove Theorem \ref{th2} in this section. We first obtain a convergence result for the noise-free iterative
sequences $\{x_n\}$ and $\{\xi_n\}$. We then consider the sequences $\{x_n^\d\}$ and $\{\xi_n^\d\}$ corresponding
to the noisy data case, and show that the discrepancy principle indeed terminates the iteration in finite steps.
We further establish a stability result which in particular implies that $x_n^\d\rightarrow x_n$ as $\d\rightarrow 0$ for
each fixed $n$. Combining all these results we finally obtain the proof of Theorem \ref{th2}.

\subsection{\bf Convergence result for noise-free case}\label{noise-free}\label{subsect4.1}

We first consider the noise-free iterative sequences $\{x_n\}$ and $\{\xi_n\}$ defined by (\ref{eq:method})
and obtain a convergence result that is crucial for proving Theorem \ref{th2}. Our proof is inspired by
\cite{HNS96,JS2012}.

\begin{theorem}\label{thm1}
Let $\X$ be reflexive and $\Y$ be uniformly smooth, let $1<r<\infty$, let $\Theta$ satisfy Assumption \ref{A0},
let $F$ satisfy Assumption \ref{A1}, and let $\{\a_n\}$ satisfy $\sum_{n=1}^\infty \a_n^{-1}=\infty$. Assume that
\begin{align}\label{21.1june}
D_{\xi_0} \Theta(x^\dag, x_0)\le \varphi (\rho).
\end{align}
Then there exists a solution $x_*$ of (\ref{1.1}) in $B_{3\rho}(x_0)\cap D(\Theta)$ such that
$$
\lim_{n\rightarrow \infty} \|x_n-x_*\|=0, \quad \lim_{n\rightarrow \infty} \Theta(x_n)=\Theta(x_*)
\quad \mbox{and} \quad \lim_{n\rightarrow \infty} D_{\xi_n} \Theta(x_*, x_n)=0.
$$
If in addition $\N(F'(x^\dag))\subset \N(F'(x))$ for all $x\in B_{3\rho}(x_0)\cap D(F)$, then $x_*=x^\dag$.
\end{theorem}

\begin{proof}
We first show by induction that for any solution $\hat{x}$ of (\ref{1.1}) in $B_{3\rho}(x_0)\cap D(\Theta)$ there holds
\begin{align}\label{7.7.1}
D_{\xi_n}\Theta(\hat{x}, x_n)\le D_{\xi_0}\Theta(\hat{x}, x_0), \quad n=0, 1, \cdots.
\end{align}
This is trivial for $n=0$. Assume that it is true for $n=m-1$ for some $m\ge 1$, we will show that it is also true for $n=m$.
From (\ref{4.3.1}) we have
$$
D_{\xi_m}\Theta(\hat{x}, x_m)-D_{\xi_{m-1}}\Theta(\hat{x}, x_{m-1}) =-D_{\xi_{m-1}}\Theta(x_m, x_{m-1})
+\l \xi_{m-1}-\xi_m, \hat{x}-x_m\r.
$$
By dropping the first term on the right which is non-positive and using the definition of $\xi_m$ we can obtain
\begin{align*}
D_{\xi_m}\Theta(\hat{x}, x_m)-D_{\xi_{m-1}}&\Theta(\hat{x}, x_{m-1})
\le \frac{1}{\a_m} \l J_r(F(x_m)-y), F'(x_m)(\hat{x}-x_m)\r.
\end{align*}
In view of the properties of the duality mapping $J_r$ it follows that
\begin{align}\label{5.4.1}
&D_{\xi_m}\Theta(\hat{x}, x_m)-D_{\xi_{m-1}}\Theta(\hat{x}, x_{m-1})\nonumber\\
&\le -\frac{1}{\a_m} \|F(x_m)-y\|^r +\frac{1}{\a_m} \|F(x_m)-y\|^{r-1} \|F(x_m)-y+F'(x_m)(\hat{x}-x_m)\|.
\end{align}
In order to proceed further, we need to show that $x_m \in B_{3\rho}(x_0)$ so that Assumption \ref{A1} (d)
on $F$ can be employed. Using the minimizing property of $x_m$, the induction hypothesis, and (\ref{21.1june}) we obtain
$$
D_{\xi_{m-1}}\Theta(x_m, x_{m-1}) \le D_{\xi_{m-1}}\Theta(x^\dag, x_{m-1})
\le D_{\xi_0} \Theta(x^\dag, x_0) \le \varphi (\rho).
$$
With the help of Assumption \ref{A0} on $\Theta$, we have
$$
\|x_m-x_{m-1}\|\le \rho, \quad \|x^\dag-x_{m-1}\|\le \rho \quad
\mbox{and} \quad \|x^\dag-x_0\|\le \rho.
$$
Therefore $x_m\in B_{3\rho}(x_0)$. Thus we may use Assumption \ref{A1} (d)
to obtain from (\ref{5.4.1}) that
\begin{equation}\label{7.7.2}
D_{\xi_m}\Theta(\hat{x}, x_m)-D_{\xi_{m-1}}\Theta(\hat{x}, x_{m-1}) \le -\frac{1-\eta}{\a_m} \|F(x_m)-y\|^r.
\end{equation}
This and the induction hypothesis imply (\ref{7.7.1}) with $n=m$.

As an immediate consequence of (\ref{7.7.1}), we know that (\ref{7.7.2}) is true for all $m$. Consequently
\begin{align}\label{5.4.0}
D_{\xi_n}\Theta(\hat{x}, x_n)\le D_{\xi_{n-1}}\Theta(\hat{x}, x_{n-1}), \quad n=1,2, \cdots
\end{align}
and
\begin{align}\label{8.4.2012}
\frac{1-\eta}{\a_n}\|F(x_n)-y\|^r \le D_{\xi_{n-1}}\Theta(\hat{x}, x_{n-1})-D_{\xi_n}\Theta(\hat{x}, x_n).
\end{align}
By using the monotonicity of $\|F(x_n)-y\|$ with respect to $n$,  we obtain
\begin{align*}
\|F(x_n)-y\|^r \sum_{j=1}^n \frac{1}{\a_j} \le \sum_{j=1}^n \frac{1}{\a_j} \|F(x_j)-y\|^r
\le \frac{1}{1-\eta} D_{\xi_0} \Theta(\hat{x}, x_0).
\end{align*}
Since $\sum_{j=1}^n \a_j^{-1}\rightarrow \infty$ as $n\rightarrow \infty$, we have $\|F(x_n)-y\|\rightarrow 0$
as $n\rightarrow \infty$.

Next we show that $\{x_n\}$ converges to a solution of (\ref{1.1}). To this end, we show that $\{x_n\}$ is a
Cauchy sequence in $\X$. For $0\le l<m<\infty$ we have from (\ref{4.3.1}) that
$$
D_{\xi_l}\Theta(x_m, x_l)=D_{\xi_l}\Theta(\hat{x}, x_l)-D_{\xi_m}\Theta(\hat{x}, x_m)+\l \xi_m-\xi_l, x_m-\hat{x}\r.
$$
By the definition of $\xi_n$ we have
\begin{align}\label{eq:19.6june}
|\l \xi_m-\xi_l, x_m- \hat{x}\r| &= \left|\sum_{n=l+1}^m \l \xi_n-\xi_{n-1}, x_m-\hat{x}\r\right| \nonumber\\
& = \left|\sum_{n=l+1}^m \frac{1}{\a_n} \l J_r(F(x_n)-y), F'(x_n)(x_m-\hat{x})\r \right|\nonumber\\
&\le \sum_{n=l+1}^m \frac{1}{\a_n} \|F(x_n)-y\|^{r-1} \|F'(x_n)(x_m-\hat{x})\|.
\end{align}
By using Assumption \ref{A1} (d) on $F$ and the monotonicity of $\|F(x_n)-y\|$ we can obtain
\begin{align}\label{eq:20.1june}
\|F'(x_n)(x_m-\hat{x})\| & \le \|F'(x_n)(x_n-\hat{x})\| +\|F'(x_n)(x_m-x_n)\| \nonumber\\
& \le (1+\eta) \left(\|F(x_n)-y\| +\|F(x_m)-F(x_n)\|\right)\nonumber\\
& \le 3(1+\eta) \|F(x_n)-y\|.
\end{align}
Therefore, by using (\ref{8.4.2012}), we have with $c_0:=3(1+\eta)/(1-\eta)$ that
\begin{align}\label{5.5.1}
|\l \xi_m-\xi_l, x_m-\hat{x}\r| & \le 3(1+\eta) \sum_{n=l+1}^m \frac{1}{\a_n} \|F(x_n)-y\|^r \nonumber\\
&\le c_0 \left(D_{\xi_l}\Theta(\hat{x}, x_l)-D_{\xi_m}\Theta(\hat{x}, x_m)\right).
\end{align}
Consequently
$$
D_{\xi_l}\Theta(x_m, x_l) \le (1+c_0) \left(D_{\xi_l}\Theta(\hat{x}, x_l)-D_{\xi_m}\Theta(\hat{x}, x_m)\right).
$$
Since $\{D_{\xi_n}\Theta(\hat{x}, x_n)\}$ is monotonically decreasing, we obtain
$D_{\xi_l}\Theta(x_m, x_l)\rightarrow 0$ as $l,m\rightarrow \infty$. In view of the uniformly convexity of $\Theta$,
we can conclude that $\{x_n\}$ is a Cauchy sequence in $\X$. Thus $x_n\rightarrow x_*$ for some $x_*\in \X$
as $n\rightarrow \infty$. Since $\|F(x_n)-y\|\rightarrow 0$ as $n\rightarrow \infty$, we may use the weakly closedness of $F$
to conclude that $x_*\in D(F)$ and $F(x_*)=y$. We remark that $x_*\in B_{3\rho}(x_0)$ because $x_n\in B_{3\rho}(x_0)$.

Next we show that
$$
x_*\in D(\Theta), \quad \lim_{n\rightarrow \infty} \Theta(x_n)=\Theta(x_*)
\quad \mbox{and} \quad \lim_{n\rightarrow \infty} D_{\xi_n} \Theta(x_*, x_n)=0.
$$
From the convexity of $\Theta$ and $\xi_n\in \p \Theta(x_n)$ it follows that
\begin{align}\label{5.5.3}
\Theta(x_n)\le \Theta(\hat{x}) +\l \xi_n, x_n-\hat{x}\r.
\end{align}
In view of (\ref{5.5.1}) we have
$$
\Theta(x_n)\le \Theta(\hat{x}) +\l \xi_0, x_n-\hat{x}\r + c_0 D_{\xi_0} \Theta(\hat{x}, x_0).
$$
Since $x_n\rightarrow x_*$ as $n\rightarrow \infty$, by using the weakly lower semi-continuity of $\Theta$ we obtain
\begin{equation}\label{5.5.2}
\Theta(x_*)\le \liminf_{n\rightarrow\infty} \Theta(x_n)\le
\Theta(\hat{x}) +\l \xi_0, x_*-\hat{x}\r + c_0 D_{\xi_0} \Theta(\hat{x}, x_0)<\infty.
\end{equation}
This implies that $x_*\in D(\Theta)$. We next use (\ref{5.5.1}) to derive for $l<n$ that
\begin{align*}
|\l\xi_n, x_n-x_*\r|
\le c_0 \left(D_{\xi_l} \Theta(x_*, x_l) -D_{\xi_n} \Theta(x_*, x_n)\right)
+|\l \xi_l, x_n-x_*\r|.
\end{align*}
By taking $n\rightarrow \infty$ and using $x_n\rightarrow x_*$ we can derive that
\begin{align*}
\limsup_{n\rightarrow \infty} |\l\xi_n, x_n-x_*\r|
\le c_0 \left(D_{\xi_l} \Theta(x_*, x_l) - \varepsilon_0\right),
\end{align*}
where $\varepsilon_0:=\lim_{n\rightarrow \infty} D_{\xi_n}\Theta(x_*, x_n)$ whose existence is guaranteed by 
the monotonicity of $\{D_{\xi_n}\Theta(x_*, x_n)\}$. Since the above inequality holds for all $l$, 
by taking $l\rightarrow \infty$ we obtain
\begin{align}\label{5.5.6.1}
\limsup_{n\rightarrow \infty} |\l\xi_n, x_n-x_*\r| \le c_0 \left(\varepsilon_0 - \varepsilon_0\right)=0.
\end{align}
Using (\ref{5.5.3}) with $\hat{x}$ replaced by $x_*$ we thus obtain
$\limsup_{n\rightarrow \infty} \Theta(x_n)\le \Theta(x_*)$.
Combining this with (\ref{5.5.2}) we therefore obtain $\lim_{n\rightarrow \infty} \Theta(x_n)=\Theta(x_*)$.
This together with (\ref{5.5.6.1}) then implies that $\lim_{n\rightarrow \infty} D_{\xi_n}\Theta(x_*, x_n)=0$.

Finally we prove $x_*=x^\dag$ under the additional condition $\N(F'(x^\dag))\subset \N(F'(x))$ for $x\in B_{3\rho}(x_0)\cap D(F)$.
We use (\ref{5.5.3}) with $\hat{x}$ replaced by $x^\dag$ to obtain
\begin{equation}\label{5.5.7}
D_{\xi_0}\Theta(x_n, x_0)\le D_{\xi_0}\Theta(x^\dag, x_0)+\l \xi_n-\xi_0, x_n-x^\dag\r.
\end{equation}
By using (\ref{5.5.1}), for any $\varepsilon>0$  we can find $l_0$ such that
$$
\left|\l \xi_n-\xi_{l_0}, x_n-x^\dag\r\right| <\frac{\varepsilon}{2}, \qquad n\ge l_0.
$$
We next consider $\l \xi_{l_0}-\xi_0, x_n-x^\dag\r$. According to the definition of $\xi_n$ we have
$\xi_j-\xi_{j-1}\in \R(F'(x_j)^*)$. Since $\X$ is reflexive and $\N(F'(x^\dag)) \subset \N(F'(x_j))$,
we have from (\ref{eq:28.1june}) that $\overline{\R(F'(x_j)^*)}\subset \overline{\R(F'(x^\dag)^*)}$.
Thus we can find $v_j\in \Y^*$ and $\beta_j \in \X^*$ such that
$$
\xi_j-\xi_{j-1}=F'(x^\dag)^* v_j +\beta_j \quad \mbox{and} \quad \|\beta_j\|
\le \frac{\varepsilon}{3 l_0 M}, \quad 1\le j\le l_0,
$$
where $M>0$ is a constant such that $\|x_n-x^\dag\|\le M$ for all $n$. Consequently
\begin{align*}
\left|\l \xi_{l_0}-\xi_0, x_n-x^\dag\r \right|&=\left|\sum_{j=1}^{l_0} \l \xi_j-\xi_{j-1}, x_n-x^\dag\r\right|\\
& =\left|\sum_{j=1}^{l_0} \left[\l v_j, F'(x^\dag) (x_n-x^\dag)\r +\l \beta_j, x_n-x^\dag\r \right]\right|\\
&\le \sum_{j=1}^{l_0} \left(\|v_j\| \|F'(x^\dag) (x_n-x^\dag)\| +\|\beta_j\| \|x_n-x^\dag\|\right)\\
&\le (1+\eta) \sum_{j=1}^{l_0} \|v_j\| \|F(x_n)-y\| +\frac{\varepsilon}{3}.
\end{align*}
Since $\|F(x_n)-y\|\rightarrow 0$ as $n\rightarrow \infty$, we can find $n_0\ge l_0$ such that
$$
|\l \xi_{l_0}-\xi_0, x_n-x^\dag\r |<\frac{\varepsilon}{2}, \qquad \forall n\ge n_0.
$$
Therefore $|\l \xi_n-\xi_0, x_n-x^\dag\r|<\varepsilon$ for all $n\ge n_0$. Since $\varepsilon>0$ is arbitrary,
we obtain $\lim_{n\rightarrow \infty} \l \xi_n-\xi_0, x_n-x^\dag\r=0$. By taking $n\rightarrow \infty$ in
(\ref{5.5.7}) and using $\Theta(x_n)\rightarrow \Theta(x_*)$ we obtain
$$
D_{\xi_0}\Theta(x_*, x_0)\le D_{\xi_0}\Theta(x^\dag, x_0).
$$
According to the definition of $x^\dag$ we must have $D_{\xi_0}\Theta(x_*, x_0)=D_{\xi_0}\Theta(x^\dag, x_0)$.
A direct application of Lemma \ref{lem existence and uniqueness of xdag} gives $x_*=x^\dag$. \hfill $\Box$
\end{proof}

As a byproduct, now we can use some estimates established in the proof of Theorem \ref{thm1}
to prove Lemma \ref{lem:A3}.

\vskip 0.15cm
\noindent
{\it Proof of Lemma \ref{lem:A3}.}
We assume that the minimization problem in (\ref{eq:method}) has two minimizers $x_n$ and $\hat{x}_n$.
Then it follows that
\begin{align*}
0 &= \frac{1}{r} \|F(\hat{x}_n)-y\|^r + \alpha_n D_{\xi_{n-1}} \Theta(\hat{x}_n,x_{n-1})
- \frac{1}{r} \|F(x_n)-y\|^r \\
& \quad \, - \alpha_n D_{\xi_{n-1}} \Theta(x_n,x_{n-1})\\
& = \Delta_r\left(F(\hat{x}_n)-y, F(x_n)-y\right)+\l J_r(F(x_n)-y), F(\hat{x}_n)-F(x_n)\r\\
&\quad+ \alpha_n\left(\Theta(\hat{x}_n)-\Theta(x_n)-\l \xi_{n-1}, \hat{x}_n-x_n\r\right).
\end{align*}
With the help of the definition of $\xi_n$ we can write
\begin{align*}
&\Theta(\hat{x}_n)-\Theta(x_n)-\l \xi_{n-1}, \hat{x}_n- x_n \r\\
& \quad  = \Theta(\hat{x}_n) - \Theta(x_n) - \l \xi_n, \hat{x}_n- x_n\r + \l \xi_n-\xi_{n-1}, \hat{x}_n- x_n\r\\
& \quad = D_{\xi_n} \Theta(\hat{x}_n, x_n) -\frac{1}{\a_n}\l J_r(F(x_n)-y), F'(x_n)(\hat{x}_n- x_n)\r.
\end{align*}
Therefore
\begin{align*}
0 &=\Delta_r\left(F(\hat{x}_n)-y,F(x_n)-y\right) +\a_n D_{\xi_n} \Theta(\hat{x}_n, x_n) \\
&\quad+ \l J_r(F(x_n)-y), F(\hat{x}_n)-F(x_n)-F'(x_n)(\hat{x}_n-x_n)\r.
\end{align*}
Since $x_n, \hat{x}_n\in B_{3\rho}(x_0)$ as shown in the proof of Theorem \ref{thm1}, we may use
Assumption \ref{A3} and the Young's inequality to obtain
\begin{align*}
0 &\geq  \Delta_r\left(F(\hat{x}_n)-y,F(x_n)-y\right) +\a_n D_{\xi_n} \Theta (\hat{x}_n, x_n) \\
& \quad \, - C_0 \|F(x_n)-y\|^{r-1} \left[D_{\xi_n} \Theta(\hat{x}_n, x_n)\right]^{1-\kappa}
\left[\Delta_r (F(\hat{x}_n)-y,F(x_n)-y)\right]^\kappa\\
&\ge \a_n D_{\xi_n} \Theta(\hat{x}_n, x_n) - (1-\kappa) \kappa^{\frac{\kappa}{1-\kappa}} C_0^{\frac{1}{1-\kappa}}
\|F(x_n)-y\|^{\frac{r-1}{1-\kappa}} D_{\xi_n} \Theta(\hat{x}_n,x_n).
\end{align*}
Recall that in the proof of Theorem \ref{thm1} we have established
$$
\|F(x_n)-y\|^r \le \frac{1}{1-\eta} s_n^{-1}  D_{\xi_0} \Theta(x^\dag, x_0) \quad \mbox{ with } s_n:=\sum_{j=1}^n \a_j^{-1}.
$$
Since $s_n^{-1} \le \min\{\a_1, \a_n\}$ and $\kappa\ge 1/r$, we therefore obtain
\begin{align*}
0 &\geq \left(1-\bar{C}_0^{\frac{1}{1-\kappa}} D_{\xi_0} \Theta(x^\dag, x_0)^{\frac{r-1}{r(1-\kappa)}} \right)
\alpha_n D_{\xi_n} \Theta(\hat{x}_n, x_n)
\end{align*}
with $\bar{C}_0:=C_0 \kappa^\kappa (1-\kappa)^{1-\kappa} (1-\eta)^{\frac{1-r}{r}}
\a_1^{\kappa-\frac{1}{r}}$. Thus we may use the second condition in (\ref{21.11june}) to conclude
that $D_{\xi_n} \Theta(\hat{x}_n, x_n)=0$ and hence $\hat{x}_n = x_n$. \hfill $\Box$

\subsection{\bf Justification of the method}

In this subsection we show that the method is well-defined, in particular we prove that,
when the data contains noise, the discrepancy principle (\ref{DP}) terminates
the iteration in finite steps, i.e. $n_\d<\infty$.

\begin{lemma}\label{lem stop}
Let $\X$ be reflexive and $\Y$ be uniformly smooth, let $\Theta$ satisfy Assumption \ref{A0}, and let $F$ satisfy
Assumption \ref{A1}. Let $1<r<\infty$ and $\tau>(1+\eta)/(1-\eta)$, and let $\{\a_n\}$ be such that
$\sum_{n=1}^\infty \a_n^{-1} =\infty$. Assume that (\ref{21.1june}) holds.
Then the discrepancy principle~\eqref{DP} terminates the iteration after $n_{\delta}<\infty$ steps.
 If $n_\d\ge 2$, then for $1\le n<n_\d$ there hold
\begin{align}
D_{\xi_{n}^{\d}} \Theta(\hat{x}, x_n^\d) &\leq D_{\xi_{n-1}^\d} \Theta(\hat{x}, x_{n-1}^\d), \label{eq decrease2}\\
\frac{1}{\alpha_n}\|F(x_n^\d)-y^\d\|^r
& \leq C_1 \left(D_{\xi_{n-1}^\d} \Theta(\hat{x}, x_{n-1}^\d)-D_{\xi_{n}^\d} \Theta(\hat{x},x_n^\d)\right).
\label{eq decrease22}
\end{align}
If, in addition, $\a_n \le c_0 \a_{n+1}$ for all $n$ with some constant $c_0>0$ and
\begin{equation}\label{eq:20.5june}
D_{\xi_0}\Theta(x^\dag, x_0) \le \frac{\tau^r-1}{\tau^r-1+c_0} \varphi (\rho),
\end{equation}
then there holds
\begin{align}\label{eq ndelta}
D_{\xi_{n_\d}^\d} \Theta(\hat{x}, x_{n_\d}^\d)\leq D_{\xi_{n_\d-1}^\d} \Theta(\hat{x}, x_{n_\d-1}^\d)
+ (1+\eta) \tau^{r-1} \frac{\d^r}{\alpha_{n_\d}},
\end{align}
where $\hat{x}$ denotes any solution of (\ref{1.1}) in $B_{3\rho}(x_0)\cap D(\Theta)$ and $C_1:= \tau/[(1-\eta)\tau-1-\eta]$.

\end{lemma}

\begin{proof}
To prove the first part, we first show by induction that
\begin{equation}\label{eq:20.4june}
x_n^\d\in B_{2\rho}(x_0) \quad \mbox{and} \quad D_{\xi_n^\d}\Theta(x^\dag, x_n^\d) \le D_{\xi_0}\Theta(x^\dag, x_0),
\quad 0\le n<n_\d.
\end{equation}
This is trivial for $n=0$. Next we assume that (\ref{eq:20.4june}) is true for $n=m-1$ for some $m<n_\d$ and
show that (\ref{eq:20.4june}) is also true for $n=m$. By the minimizing property of $x_m^\d$ and
the induction hypothesis we have
\begin{align}\label{eq7.21}
\frac{1}{r} \|F(x_m^\d)-y^\d\|^r +\a_m D_{\xi_{m-1}^\d} \Theta(x_m^\d, x_{m-1}^\d)
& \le \frac{1}{r} \d^r + \a_m D_{\xi_{m-1}^\d} \Theta(x^\dag, x_{m-1}^\d) \nonumber\\
& \le \frac{1}{r} \d^r + \a_m D_{\xi_0}\Theta(x^\dag, x_0).
\end{align}
Since $\|F(x_m^\d)-y^\d\|>\tau \d$, we can obtain
\begin{align*}
\frac{\tau^r}{r} \d^r +\a_m D_{\xi_{m-1}^\d} \Theta(x_m^\d, x_{m-1}^\d)
 \le \frac{1}{r} \d^r + \a_m D_{\xi_0}\Theta(x^\dag, x_0).
\end{align*}
Because $\tau>1$, this implies that
\begin{equation}\label{eq:20.6june}
\a_m\ge \frac{(\tau^r-1) \d^r}{r D_{\xi_0} \Theta(x^\dag, x_0)} \quad \mbox{and} \quad
D_{\xi_{m-1}^\d} \Theta(x_m^\d, x_{m-1}^\d) \le  D_{\xi_0}\Theta(x^\dag, x_0).
\end{equation}
By Assumption \ref{A0} and the condition (\ref{21.1june}), we can derive that
$\|x_m^\d-x_{m-1}^\d\| \le \rho$. In view of the induction hypothesis we also have $\|x_{m-1}^\d-x_0\|
\le 2\rho$. Thus $x_m^\d\in B_{3\rho}(x_0)$.

We are now able to use Assumption \ref{A1} (d) and the similar argument for deriving \eqref{5.4.1} to obtain that
\begin{align}
D_{\xi_m^\d} & \Theta (\hat{x},x_m^\d)- D_{\xi_{m-1}^\d} \Theta(\hat{x},x_{m-1}^\d) \nonumber\\
&\leq \langle \xi_m^\d-\xi_{m-1}^\d, x_m^\d-\hat{x}\rangle
 = -\frac{1}{\alpha_m} \langle J_r(F(x_m^\d)-y^\d), F'(x_m^\d)(x_m^\d-\hat{x}) \rangle \nonumber\\
&\leq -\frac{1}{\alpha_m}\|F(x_m^{\delta})-y^{\delta}\|^{r}
 +\frac{1}{\alpha_m} \|F(x_m^\d)-y^\d\|^{r-1} \left(\d + \eta \|F(x_m^\d) -y\|\right) \nonumber\\
& \leq -\frac{1-\eta}{\alpha_m} \|F(x_m^\d)-y^\d\|^{r}
 +\frac{1+\eta}{\a_m} \|F(x_m^\d)-y^\d\|^{r-1} \d. \label{eq:19.5june}
\end{align}
Using again $\|F(x_m^\d)-y^\d\|>\tau\d$, we can conclude that
\begin{align}\label{20.6june}
D_{\xi_m^\d} \Theta (\hat{x}, x_m^\d)- D_{\xi_{m-1}^\d} \Theta(\hat{x}, x_{m-1}^\d)
&\leq -\frac{1}{\alpha_m} \left(1-\eta-\frac{1+\eta}{\tau}\right) \|F(x_m^\d)-y^\d\|^r.
\end{align}
Since $\tau>(1+\eta)/(1-\eta)$, we obtain
$$
D_{\xi_m^\d}\Theta(\hat{x}, x_m^\d) \le D_{\xi_{m-1}^\d}\Theta(\hat{x}, x_{m-1}^\d).
$$
In view of this inequality with $\hat{x}=x^\dag$ and the induction hypothesis, we obtain the second result
in (\ref{eq:20.4june}) with $n=m$. By using again Assumption \ref{A0} and (\ref{21.1june})
we have $\|x_m^\d-x^\dag\|\le \rho$ and $\|x^\dag-x_0\|\le \rho$ which imply that $x_m^\d\in B_{2\rho}(x_0)$.
We therefore complete the proof of (\ref{eq:20.4june}). As a direct consequence, we can see that
(\ref{20.6june}) holds for all $1\le m<n_\d$ which implies (\ref{eq decrease2}) and (\ref{eq decrease22}).

In view of~\eqref{eq decrease22} and the monotonicity (\ref{mono}) of $\|F(x_n^\d)-y^\d\|$
with respect to $n$, it follows that
\begin{align*}
\|F(x_n^\d)-y^\d\|^r \sum_{j=1}^n\frac{1}{\alpha_j}
 \leq \sum_{j=1}^n\frac{1}{\alpha_j} \|F(x_j^\d)-y^\d\|^r
\leq \frac{\tau}{(1-\eta)\tau-1-\eta} D_{\xi_0} \Theta(\hat{x}, x_0).
\end{align*}
Since $\|F(x_n^\d)-y^\d\|>\tau\d$ for $1\leq n<n_\d$ and $\sum_{j=1}^n\alpha_j^{-1}\rightarrow\infty$
as $n\rightarrow\infty$, we can conclude that $n_\d$ is a finite integer.

Finally we prove the second part, i.e. the inequality (\ref{eq ndelta}). Since (\ref{eq7.21}) is true for $m=n_\d$, we have
$$
D_{\xi_{n_\d-1}^\d} \Theta(x_{n_\d}^\d, x_{n_\d-1}^\d) \le \frac{\d^r}{r \a_{n_\d}}
+D_{\xi_0} \Theta(x^\dag, x_0).
$$
Recall from (\ref{eq:20.6june}) that $\a_{n_\d-1} \ge (\tau^r-1) \d^r /(r D_{\xi_0}\Theta(x^\dag, x_0))$.
Since $\a_{n_\d-1}\le c_0 \a_{n_\d}$, we can derive that
$$
D_{\xi_{n_\d-1}^\d} \Theta(x_{n_\d}^\d, x_{n_\d-1}^\d) \le
\frac{\tau^r-1+c_0}{\tau^r-1} D_{\xi_0} \Theta(x^\dag, x_0).
$$
It then follows from Assumption \ref{A0} and (\ref{eq:20.5june}) that $\|x_{n_\d}^\d-x_{n_\d-1}^\d\|\le \rho$.
Since $x_{n_\d-1}^\d\in B_{2\rho}(x_0)$ we obtain $x_{n_\d}^\d\in B_{3\rho}(x_0)$. Thus we can employ
Assumption \ref{A1} (d) to conclude that (\ref{eq:19.5june}) is also true for $m=n_\d$.
By setting $m=n_\d$ in (\ref{eq:19.5june}) and using $\|F(x_{n_\d}^\d)-y^\d\|\leq \tau\d$,
we can obtain \eqref{eq ndelta}. \hfill $\Box$
\end{proof}

As a byproduct of the proof of Lemma \ref{lem stop}, we have the following result which will be used to show
$\lim_{\d\rightarrow 0} \Theta(x_{n_\d}^\d) =\Theta(x_*)$ in the proof of Theorem \ref{th2}.

\begin{lemma}\label{lem 19june}
Let all the conditions in Lemma \ref{lem stop} hold, and let $\hat{x}$ be any solution
of (\ref{1.1}) in $B_{3\rho}(x_0)\cap D(\Theta)$. Then for all $0\le l< n_\d$ there holds
\begin{align}\label{eq:19.9june}
\left|\l \xi_{n_\d}^\d-\xi_l^\d, \hat{x}-x_{n_\d}^\d\r \right| \le C_2 \frac{\d^r}{\a_{n_\d}}
+ C_3 D_{\xi_l^\d} \Theta (\hat{x}, x_l^\d),
\end{align}
where $C_2:= 3(1+\eta) \tau^{r-1} (1+\tau)$ and $C_3:=3(1+\eta) (1+\tau)/[(1-\eta) \tau-1-\eta]$.
\end{lemma}

\begin{proof}
By the definition of $\xi_n^\d$ and the property of the duality mapping $J_r$, we can obtain, using the
similar argument for deriving (\ref{eq:19.6june}), that
\begin{align*}
\left|\l \xi_{n_\d}^\d-\xi_l^\d, \hat{x}-x_{n_\d}^\d\r \right|
\le \sum_{n=l+1}^{n_\d} \frac{1}{\a_n} \|F(x_n^\d)-y^\d\|^{r-1}\|F'(x_n^\d) (\hat{x}-x_{n_\d}^\d)\|.
\end{align*}
With the help of Assumption \ref{A1} (d)  and the monotonicity (\ref{mono}) of $\|F(x_n^\d)-y^\d\|$ with respect to $n$,
similar to the derivation of (\ref{eq:20.1june}) we have for $n\le n_\d$ that
$$
\|F'(x_n^\d)(\hat{x}-x_{n_\d}^\d)\| \le 3(1+\eta) \left(\|F(x_n^\d)-y^\d\| +\d\right).
$$
Therefore
\begin{align*}
\left|\l \xi_{n_\d}^\d-\xi_l^\d, \hat{x}-x_{n_\d}^\d\r \right|
\le 3(1+\eta) \sum_{n=l+1}^{n_\d} \frac{1}{\a_n} \|F(x_n^\d)-y^\d\|^{r-1} \left(\|F(x_n^\d)-y^\d\| +\d\right).
\end{align*}
Since $\|F(x_{n_\d}^\d)-y^\d\| \le \tau \d$ and $\|F(x_n^\d)-y^\d\| >\tau \d$ for $0\le n <n_\d$, we thus obtain
\begin{align} \label{eq:19.7june}
& \left|\l \xi_{n_\d}^\d-\xi_l^\d, \hat{x}-x_{n_\d}^\d\r \right| \nonumber\\
& \quad \le 3(1+\eta)\tau^{r-1} (1+\tau) \frac{\d^r}{\a_{n_\d}}
+ \frac{3(1+\eta)(1+\tau)}{\tau} \sum_{n=l+1}^{n_\d-1} \frac{1}{\a_n} \|F(x_n^\d)-y^\d\|^r.
\end{align}
In view of (\ref{eq decrease22}) in Lemma \ref{lem stop}, we can see that
$$
\sum_{n=l+1}^{n_\d-1} \frac{1}{\a_n} \|F(x_n^\d)-y^\d\|^r
\le \frac{\tau}{(1-\eta)\tau-1-\eta} D_{\xi_l^\d}\Theta (\hat{x}, x_l^\d).
$$
Combining this inequality with (\ref{eq:19.7june}) gives the desired estimate. \hfill $\Box$
\end{proof}

\subsection{\bf Stability}

We will prove some stability results on the method which connect $\{x_n^\d\}$ with $\{x_n\}$.
These results enable us to use Theorem \ref{thm1} to complete the proof of Theorem \ref{th2}.

\begin{lemma}\label{lem xdelta}
Let $\X$ be reflexive and $\Y$ be uniformly smooth, let $\Theta$ satisfy Assumption \ref{A0},
and let $F$ satisfy Assumptions \ref{A1} and \ref{A2}. Then for each fixed $n$ there hold
\begin{align}\label{eq 3convergence}
x_n^\d\rightarrow x_n,\quad \Theta(x_n^\d)\rightarrow \Theta(x_n)\quad \mbox{and} \quad \xi_n^\d\rightarrow \xi_n
\end{align}
as $y^\d\rightarrow y$.
\end{lemma}

\begin{proof}
We show this result by induction. It is trivial when $n=0$ since $x_0^\d = x_0$ and $\xi_0^\d = \xi_0$.
In the following we assume that the result is proved for $n=m-1$ and show that the result holds also for $n=m$.

We will adapt the argument from \cite{EKN89}. Let $\{y^{\d_i}\}$ be a sequence of data satisfying
$\|y^{\d_i}-y\|\le \d_i$ with $\d_i\rightarrow 0$. By the minimizing property of $x_m^{\d_i}$ we have
\begin{align*}
\frac{1}{r}\|F(x_m^{\d_i})-y^{\d_i}\|^r
+ \alpha_m D_{\xi_{m-1}^{\d_i}} \Theta (x_m^{\d_i},x_{m-1}^{\d_i})
\leq \frac{1}{r}\|F(x_{m-1}^{\d_i})-y^{\d_i}\|^r.
\end{align*}
By the induction hypothesis, we can see that the right hand side of the above inequality
is uniformly bounded with respect to $i$. Therefore both $\{\|F(x_m^{\d_i})-y^{\d_i}\|\}$ and
$\{D_{\xi_{m-1}^{\d_i}}\Theta(x_m^{\d_i},x_{m-1}^{\d_i})\}$ are uniformly bounded
with respect to $i$. Consequently $\{F(x_m^{\d_i})\}$ is bounded in $\Y$ and $\{x_m^{\d_i}\}$
is bounded in $\X$; here we used the uniformly convexity of $\Theta$. Since both $\X$ and $\Y$
are reflexive, by taking a subsequence if necessary,
we may assume that $x_m^{\d_i}\rightharpoonup \bar{x}_m\in\X$ and
$F(x_m^{\d_i})\rightharpoonup \bar{y}_m\in \Y$ as $i\rightarrow\infty$. Since $F$ is weakly closed,
we have $\bar{x}_m \in D(F)$ and $F(\bar{x}_m)=\bar{y}_m$. In view of the weakly lower semi-continuity
of Banach space norm we have
\begin{align}\label{eq use2}
\|F(\bar{x}_m)-y\|\leq \liminf_{i\rightarrow\infty}\|F(x_m^{\d_i})-y^{\d_i}\|.
\end{align}
Moreover, by using $x_m^{\d_i}\rightharpoonup \bar{x}_m$, the weakly lower semi-continuity of $\Theta$,
and the induction hypothesis, we have
\begin{align}\label{eq use1}
\liminf_{i\rightarrow\infty} D_{\xi_{m-1}^{\d_i}} \Theta(x_m^{\d_i},x_{m-1}^{\d_i})
& = \liminf_{i\rightarrow\infty} \Theta(x_m^{\d_i}) - \Theta(x_{m-1}) - \langle\xi_{m-1},\bar{x}_m-x_{m-1}\rangle \nonumber\\
&\geq \Theta(\bar{x}_m) - \Theta(x_{m-1}) - \langle\xi_{m-1},\bar{x}_m-x_{m-1}\rangle \nonumber\\
&= D_{\xi_{m-1}}\Theta(\bar{x}_m,x_{m-1}).
\end{align}
The inequalities (\ref{eq use2}) and (\ref{eq use1}) together with the minimizing property
of $x_m^{\d_i}$ and the induction hypothesis imply
\begin{align*}
\frac{1}{r}  \|F(\bar{x}_m) & -y\|^r + \alpha_m D_{\xi_{m-1}} \Theta(\bar{x}_m,x_{m-1})\\
&\leq \liminf_{i\rightarrow\infty} \left\{\frac{1}{r}\|F(x_m^{\d_i})-y^{\d_i}\|^r
+ \alpha_m D_{\xi_{m-1}^{\d_i}} \Theta (x_m^{\d_i},x_{m-1}^{\d_i})\right\}\\
&\leq \limsup_{i\rightarrow\infty} \left\{\frac{1}{r}\|F(x_m^{\d_i})-y^{\d_i}\|^r
+ \alpha_m D_{\xi_{m-1}^{\d_i}} \Theta (x_m^{\d_i},x_{m-1}^{\d_i})\right\}\\
&\leq \limsup_{i\rightarrow\infty} \left\{\frac{1}{r}\|F(x_m)-y^{\d_i}\|^r
+ \alpha_m D_{\xi_{m-1}^{\d_i}} \Theta (x_m,x_{m-1}^{\d_i})\right\}\\
&= \frac{1}{r}\|F(x_m)-y\|^r + \alpha_m D_{\xi_{m-1}} \Theta(x_m,x_{m-1}).
\end{align*}
According to the definition of $x_m$ and Assumption \ref{A2}, we must have $\bar{x}_m = x_m$. Therefore
$x_m^{\d_i}\rightharpoonup x_m$, $F(x_m^{\d_i})\rightharpoonup F(x_m)$, and
\begin{align}\label{eq:18june}
\lim_{i\rightarrow\infty} & \left\{\frac{1}{r}\|F(x_m^{\d_i})-y^{\d_i}\|^r
+ \alpha_m D_{\xi_{m-1}^{\d_i}} \Theta (x_m^{\d_i},x_{m-1}^{\d_i})\right\} \nonumber\\
&\qquad \qquad = \frac{1}{r}\|F(x_m)-y\|^r + \alpha_m D_{\xi_{m-1}} \Theta(x_m,x_{m-1}).
\end{align}

Next we will show that
\begin{align}\label{eq lim}
\lim_{i\rightarrow\infty} D_{\xi_{m-1}^{\d_i}} \Theta(x_m^{\d_i},x_{m-1}^{\d_i}) = D_{\xi_{m-1}} \Theta(x_m,x_{m-1}).
\end{align}
Let
\begin{align*}
a: =  \limsup_{i\rightarrow\infty} D_{\xi_{m-1}^{\d_i}} \Theta(x_m^{\d_i},x_{m-1}^{\d_i})\quad \textrm{and} \quad
b: = D_{\xi_{m-1}} \Theta(x_m,x_{m-1}).
\end{align*}
In view of \eqref{eq use1}, it suffices to show $a\leq b$. Assume to the contrary that $a>b$. By taking a subsequence
if necessary, we may assume that
\begin{align*}
a= \lim_{i\rightarrow\infty} D_{\xi_{m-1}^{\d_i}} \Theta(x_m^{\d_i},x_{m-1}^{\d_i}).
\end{align*}
It then follows from (\ref{eq:18june}) that
\begin{align*}
\frac{1}{r} \lim_{i\rightarrow\infty}\|F(x_m^{\d_i})-y^{\d_i}\|^r
=\frac{1}{r} \|F(x_m)-y\|^r +\alpha_m (b-a) < \frac{1}{r}\|F(x_m)-y\|^r
\end{align*}
which is a contradiction to (\ref{eq use2}).  We therefore obtain (\ref{eq lim}).

By using the induction hypothesis and $x_m^{\d_i}\rightharpoonup x_m$, we obtain from (\ref{eq lim}) that
\begin{align*}
\lim_{i\rightarrow\infty} \Theta(x_m^{\d_i}) = \Theta(x_m).
\end{align*}
Since $x_m^{\d_i}\rightharpoonup x_m$ and since $\Theta$ has the Kadec property, see Lemma \ref{lem:Kadec},
we obtain that $x_m^{\d_i}\rightarrow x_m$ as $i\rightarrow\infty$.
Finally, from the definition of $\xi_m^{\d_i}$, the induction hypothesis, and the continuity
of the map $x\to F'(x)$, and the continuity of the duality mapping $J_r$, it follows that
$\xi_m^{\d_i}\rightarrow \xi_m$ as $i\rightarrow\infty$.

The above argument shows that for any sequence $\{y^{\d_i}\}$ converging to $y$, the sequence $\{x_m^{\d_i}\}$
always has a subsequence, still denoted as $x_m^{\d_i}$, such that $x_m^{\d_i}\rightarrow x_m$,
$\Theta(x_m^{\d_i}) \rightarrow \Theta(x_m)$ and $\xi_m^{\d_i} \rightarrow \xi_m$ as $i\rightarrow \infty$.
Therefore, we obtain (\ref{eq 3convergence}) with $n=m$ as $y^\d \rightarrow y$. The proof is complete. \hfill $\Box$
\end{proof}

\subsection{\bf Proof of Theorem \ref{th2}}

Since other parts have been proved in Lemma \ref{lem stop}, it remains only to show the convergence
result (\ref{eq:convergence}), where $x_*$ is the limit of $\{x_n\}$ which exists by Theorem \ref{thm1}.

Assume first that $\{y^{\d_i}\}$ is a sequence satisfying $\|y^{\d_i}-y\|\leq \d_i$ with $\d_i\rightarrow 0$
such that $n_{\d_i}\rightarrow n_0$ as $i\rightarrow\infty$ for some integer $n_0$. We may assume $n_{\d_i} =n_0$
for all $i$. From the definition of $n_{\d_i} = n_0$, we have
\begin{align*}
\|F(x_{n_0}^{\d_i}) - y^{\d_i}\|\leq\tau\d_i.
\end{align*}
Since Lemma \ref{lem xdelta} implies $x_{n_0}^{\d_i}\rightarrow x_{n_0}$, by letting $i\rightarrow \infty$
we have $F(x_{n_0}) = y$. This together with the definition of $x_n$ implies that $x_n=x_{n_0}$ for
all $n\geq n_0$. Since Theorem \ref{thm1} implies $x_n\rightarrow x_*$ as $n\rightarrow\infty$,
we must have $x_{n_0} = x_*$. Consequently, we have from Lemma \ref{lem xdelta} that
$x_{n_{\d_i}}^{\d_i} \rightarrow x_*$, $\Theta(x_{n_{\d_i}}^{\d_i})=\Theta(x_{n_0}^{\d_i})\rightarrow
\Theta(x_{n_0})=\Theta(x_*)$ and
\begin{align*}
D_{\xi_{n_{\d_i}}^{\d_i}} \Theta(x_*, x_{n_{\d_i}}^{\d_i})
= D_{\xi_{n_0}^{\d_i}} \Theta(x_{n_0},x_{n_0}^{\d_i})\rightarrow 0
\end{align*}
as $i\rightarrow \infty$.

Assume next that $\{y^{\d_i}\}$ is a sequence satisfying $\|y^{\d_i}-y\|\leq \d_i$ with $\d_i\rightarrow 0$
such that $n_i:=n_{\d_i} \rightarrow\infty$ as $i\rightarrow \infty$. We first show that
\begin{align}\label{eq:19june}
D_{\xi_{n_i-2}^{\d_i}} \Theta(x_*, x_{n_i-2}^{\d_i})\rightarrow 0\qquad\textrm{as } i\rightarrow\infty.
\end{align}
Let $\epsilon>0$ be an arbitrary number. Since Theorem \ref{thm1} implies $D_{\xi_n} \Theta(x_*,x_n)\rightarrow 0$
as $n\rightarrow \infty$, there exists an integer $n(\epsilon)$ such that
$D_{\xi_{n(\epsilon)}}\Theta(x_*, x_{n(\epsilon)})<\epsilon/2$. On the other hand,
since Lemma \ref{lem xdelta} implies $x_{n(\epsilon)}^{\d_i}\rightarrow x_{n(\epsilon)}$,
$\Theta(x_{n(\epsilon)}^{\d_i}) \rightarrow \Theta(x_{n(\epsilon)})$ and
$\xi_{n(\epsilon)}^{\d_i}\rightarrow\xi_{n(\epsilon)}$ as $i\rightarrow\infty$, we can
pick an integer $i(\epsilon)$ large enough such that for all $i\ge i(\epsilon)$ there hold
$n_i-2\geq n(\epsilon)$ and
\begin{align*}
\left|D_{\xi_{n(\epsilon)}^{\d_i}} \Theta(x_*, x_{n(\epsilon)}^{\d_i})
- D_{\xi_{n(\epsilon)}} \Theta(x_*, x_{n(\epsilon)})\right|<\frac{\epsilon}{2}.
\end{align*}
Therefore, it follows from Lemma \ref{lem stop} that
\begin{align*}
D_{\xi_{n_i-2}^{\d_i}} \Theta(x_*,x_{n_i-2}^{\d_i})
&\leq D_{\xi_{n(\epsilon)}^{\d_i}} \Theta(x_*, x_{n(\epsilon)}^{\d_i})
\leq D_{\xi_{n(\epsilon)}} \Theta(x_*, x_{n(\epsilon)})
+ \frac{\epsilon}{2}<\epsilon
\end{align*}
for all $i\geq i(\epsilon)$. Since $\epsilon>0$ is arbitrary, we thus obtain (\ref{eq:19june}).
With the help of \eqref{eq decrease2},
we then obtain
\begin{align}\label{eq:20.2june}
D_{\xi_{n_i-1}^{\d_i}} \Theta(x_*, x_{n_i-1}^{\d_i})\rightarrow 0 \qquad\textrm{as } i\rightarrow\infty.
\end{align}
In view of \eqref{eq decrease22} we have
\begin{align*}
\frac{1}{\alpha_{n_i-1}}\|F(x_{n_i-1}^{\d_i})-y^{\d_i}\|^r
\leq \frac{\tau}{(1-\eta)\tau-1-\eta} D_{\xi_{n_i-2}^{\d_i}} \Theta(x_*, x_{n_i-2}^{\d_i}).
\end{align*}
Since $\|F(x_{n_i-1}^{\d_i})-y^{\d_i}\|>\tau\d_i$, we can conclude from (\ref{eq:19june}) that
$\d_i^r/\alpha_{n_i-1}\rightarrow 0$.
Since $\alpha_{n_i-1}\leq c_0\alpha_{n_i}$, we must have $\delta_i^r/\alpha_{n_i}\rightarrow 0$
as $i\rightarrow \infty$. In view of \eqref{eq ndelta} and (\ref{eq:20.2june}), we can obtain
\begin{align}\label{eq:19.8june}
D_{\xi_{n_i}^{\d_i}} \Theta(x_*, x_{n_i}^{\d_i})\rightarrow 0
\qquad \mbox{ as } i\rightarrow\infty,
\end{align}
which together with the uniformly convexity of $\Theta$ implies that $x_{n_i}^{\d_i} \rightarrow x_*$
as $i\rightarrow \infty$.

Finally we show that $\Theta(x_{n_i}^{\d_i}) \rightarrow \Theta(x_*)$ as $i \rightarrow \infty$.
In view of (\ref{eq:19.8june}), it suffices to show that
\begin{align}\label{eq:19.11june}
\l \xi_{n_i}^{\d_i}, x_*-x_{n_i}^{\d_i}\r \rightarrow 0 \qquad \mbox{as } i \rightarrow \infty.
\end{align}
Recall that $\Theta(x_n)\rightarrow \Theta(x_*)$ and $\l \xi_n, x_*-x_n\r \rightarrow 0$ as
$n \rightarrow \infty$ which have been established in Theorem \ref{thm1} and its proof. Thus, for any $\epsilon>0$, we
can pick an integer $l_0$ such that
\begin{align}\label{eq:19.10june}
\left|\Theta(x_{l_0})-\Theta(x_*)\right| <\epsilon \quad \mbox{and} \quad
\left|\l \xi_{l_0}, x_*-x_{l_0}\r\right| <\epsilon.
\end{align}
Then, using (\ref{eq:19.9june}) in Lemma \ref{lem 19june}, we can derive
\begin{align*}
\left|\l \xi_{n_i}^{\d_i}, x_*-x_{n_i}^{\d_i}\r\right|
& \le \left|\l \xi_{l_0}^{\d_i}, x_*-x_{n_i}^{\d_i}\r \right|
+\left|\l \xi_{n_i}^{\d_i} -\xi_{l_0}^{\d_i}, x_*-x_{n_i}^{\d_i}\r\right| \\
& \le  \left|\l \xi_{l_0}^{\d_i}, x_*-x_{n_i}^{\d_i}\r \right|
+ C_2 \frac{\d_i^r}{\a_{n_i}} + C_3 D_{\xi_{l_0}^{\d_i}} \Theta (x_*, x_{l_0}^{\d_i}).
\end{align*}
By using the definition of Bregman distance and (\ref{eq:19.10june}) we have
\begin{align*}
D_{\xi_{l_0}^{\d_i}} \Theta(x_*, x_{l_0}^{\d_i})
&= \left[ \Theta(x_*)-\Theta(x_{l_0}) \right] + \left[ \Theta(x_{l_0}) -\Theta(x_{l_0}^{\d_i})\right]
 - \l \xi_{l_0}, x_*-x_{l_0}\r \\
& \quad \, -\l \xi_{l_0}, x_{l_0} -x_{l_0}^{\d_i}\r
 - \l \xi_{l_0}^{\d_i}-\xi_{l_0}, x_*-x_{l_0}^{\d_i}\r\\
&\le 2\epsilon + \left| \Theta(x_{l_0}) -\Theta(x_{l_0}^{\d_i})\right| +
\left|\l \xi_{l_0}, x_{l_0} -x_{l_0}^{\d_i}\r\right| + \left| \l \xi_{l_0}^{\d_i}-\xi_{l_0}, x_*-x_{l_0}^{\d_i}\r \right|.
\end{align*}
Therefore
\begin{align*}
\left|\l \xi_{n_i}^{\d_i}, x_*-x_{n_i}^{\d_i}\r\right|
& \le 2C_3 \epsilon + C_2 \frac{\d_i^r}{\a_{n_i}} +  \left|\l \xi_{l_0}^{\d_i}, x_*-x_{n_i}^{\d_i}\r \right|
+ C_3  \left| \Theta(x_{l_0}) -\Theta(x_{l_0}^{\d_i})\right| \\
& \quad \, + C_3 \left|\l \xi_{l_0}, x_{l_0} -x_{l_0}^{\d_i}\r\right|
+ C_3 \left| \l \xi_{l_0}^{\d_i}-\xi_{l_0}, x_*-x_{l_0}^{\d_i}\r \right|.
\end{align*}
In view of Lemma \ref{lem xdelta} and the facts that $\d_i^r/\a_{n_i}\rightarrow 0$ and $x_{n_i}^{\d_i} \rightarrow x_*$
as $i\rightarrow \infty$ which we have established in the above, we can conclude that there is an integer
$i_0(\epsilon)$ such that for all $i>i_0(\epsilon)$ there hold $n_i>l_0$ and
$\left|\l \xi_{n_i}^{\d_i}, x_*-x_{n_i}^{\d_i}\r\right| \le 3 C_3 \epsilon$.
Since $\epsilon>0$ is arbitrary, we thus obtain (\ref{eq:19.11june}).

\subsection{\bf A variant of the discrepancy principle}

When $n_\d$ denotes the integer determined by the discrepancy principle (\ref{DP}), from Lemma \ref{lem stop}
we can see that the Bregman distance $D_{\xi_n^\d}\Theta(x^\dag, x_n^\d)$ is decreasing up to $n=n_\d-1$.
This monotonicity, however, may not hold at $n=n_\d$. Therefore, it seems reasonable to consider
the following variant of the discrepancy principle.

\begin{Rule}\label{rule4.1}
Let $\tau>1$ be a given number. If $\|F(x_0)-y^\d\| \le \tau \d$, we define $n_\d:=0$; otherwise we define
\begin{align*}
n_\d:= \max \left\{ n: \|F(x_n^\d)-y^\d\|\ge \tau \d \right\},
\end{align*}
i.e., $n_\d$ is the integer such that
\begin{align*}
\|F(x_{n_\d+1}^\d)-y^\d\|<\tau\d\leq \|F(x_n^\d)-y^\d\|,
\quad 0\leq n\leq n_\d.
\end{align*}
\end{Rule}

We point out that the argument for proving Theorem \ref{th2} can be used to prove the convergence property of $x_{n_\d}^\d$ for
$n_\d$ determined by Rule \ref{rule4.1}, we can even drop the condition $\alpha_n\leq c_0\alpha_{n+1}$
on $\{\alpha_n\}$ in Theorem \ref{th2}. In fact we have the following result.

\begin{theorem}\label{th4}
Let $\X$ be reflexive and $\Y$ be uniformly smooth, $\Theta$ satisfy Assumption \ref{A0}, and $F$ satisfy
Assumptions \ref{A1} and \ref{A2}. Let $1<r<\infty$ and $\tau>(1+\eta)/(1-\eta)$, and
let $\{\alpha_n\}$ be such that $\sum_{n=1}^{\infty}\alpha_n^{-1} = \infty$. Assume further that
\begin{equation*}
D_{\xi_0} \Theta(x^\dag, x_0) \le  \varphi (\rho).
\end{equation*}
Then, the integer $n_\d$ defined by Rule \ref{rule4.1} is finite. Moreover, there is a solution $x_*\in D(\Theta)$ of (\ref{1.1}) such that
\begin{align} \label{eq:convergence222}
x_{n_\d}^\d\rightarrow x_*, \quad  \Theta(x_{n_\d}^\d) \rightarrow \Theta(x_*) \quad
\mbox{and} \quad D_{\xi_{n_{\delta}}^{\delta}} \Theta(x_*, x_{n_{\delta}}^{\delta}) \rightarrow 0
\end{align}
as $\d\rightarrow 0$. If, in addition, $\N(F'(x^\dag))\subset \N (F'(x))$ for all $x\in B_{3\rho}(x_0)\cap D(F)$, then
$x_*=x^\dag$.
\end{theorem}

\begin{proof}
The proof of Lemma \ref{lem stop} can be used without change to show that $n_\d<\infty$ and that
(\ref{eq decrease2}) and (\ref{eq decrease22}) hold for $1\le n\le n_\d$. Consequently, (\ref{eq:19.9june})
in Lemma \ref{lem 19june} becomes
\begin{align}\label{eq:22.1june}
\left|\l \xi_{n_\d}^\d-\xi_l^\d, x_*-x_{n_\d}^\d\r \right|
\le \frac{3(1+\eta)(1+\tau)}{(1-\eta)\tau-1-\eta} D_{\xi_l^\d} \Theta (x_*, x_l^\d), \quad 0\le l<n_\d.
\end{align}

In order to prove the convergence result (\ref{eq:convergence222}), as in the proof of Theorem \ref{th2}
we consider two cases.

Assume first that $\{y^{\d_i}\}$ is a sequence satisfying $\|y^{\d_i}-y\|\leq \d_i$ with $\d_i\rightarrow 0$
such that $n_{\d_i}\rightarrow n_0$ as $i\rightarrow\infty$ for some integer $n_0$. We may assume $n_{\d_i} =n_0$
for all $i$. By Rule \ref{rule4.1} we always have $\|F(x_{n_0+1}^{\d_i})-y^{\d_i}\|\le \tau\d_i$. By letting $i\rightarrow\infty$, we obtain
$F(x_{n_0+1}) = y$. This together with the definition of $x_n$ implies that $x_n = x_{n_0+1}$ for all $n\geq n_0+1$.
It then follows from Theorem \ref{thm1} that $x_* = x_{n_0+1}$. We claim that
$x_{n_0+1} = x_{n_0}$. To see this, by using the definition of $\xi_{n_0+1}$, we have
\begin{align*}
\xi_{n_0+1} = \xi_{n_0} -\frac{1}{\alpha_{n_0+1}} F'(x_{n_0+1})^* J_r(F(x_{n_0+1})-y) = \xi_{n_0}.
\end{align*}
Therefore
\begin{align*}
D_{\xi_{n_0}} \Theta(x_{n_0+1},x_{n_0}) &\le D_{\xi_{n_0}} \Theta(x_{n_0+1},x_{n_0})
+D_{\xi_{n_0+1}} \Theta(x_{n_0},x_{n_0+1}) \\
&=\l \xi_{n_0+1} -\xi_{n_0}, x_{n_0+1}-x_{n_0}\r =0.
\end{align*}
This and the strictly convexity of $\Theta$ imply that $x_{n_0+1} = x_{n_0}$. Consequently $x_{n_0} = x_*$.
A simple application of Lemma \ref{lem xdelta} then gives the desired conclusion.

Assume next that $\{y^{\d_i}\}$ is a sequence satisfying $\|y^{\d_i}-y\|\leq\d_i$ with $\d_i\rightarrow 0$
such that $n_{\d_i}\rightarrow\infty$ as $i\rightarrow\infty$. We can follow the argument for deriving
(\ref{eq:19june}) to show that $D_{\xi_{n_i}^{\d_i}} \Theta(x_*, x_{n_i}^{\d_i})\rightarrow 0$ which in
turn implies that $x_{n_i}^{\d_i} \rightarrow x_*$ by the uniformly convexity of $\Theta$.
Then we can use (\ref{eq:22.1june}) and follow the same procedure in the proof of Theorem \ref{th2}
to obtain $\Theta(x_{n_i}^{\d_i})\rightarrow \Theta(x_*)$ as $i\rightarrow \infty$. \hfill $\Box$
\end{proof}

\section{Numerical examples}\label{se5}

In this section we present some numerical simulations to test the performance of our method by considering
a linear integral equation of the first kind and a nonlinear problem arising from the parameter
identification in partial differential equations.

\begin{example}\label{ex1}

We consider the linear integral equation of the form
\begin{align}\label{linear operator}
Ax(s) := \int_0^1 K(s,t) x(t)dt = y(s) \qquad\textrm{on}\ [0,1],
\end{align}
where
\begin{align*}
K(s,t) = \left\{ \begin{array}{lll}
    40s(1-t), & \quad s\leq t, \\
    40t(1-s), & \quad s\geq t.
  \end{array}
\right.
\end{align*}
It is clear that $A: \X:=L^2[0,1]\to \Y:=L^2[0,1]$ is a compact operator. Our goal is to find the solution
of \eqref{linear operator} by using some noisy data $y^{\delta}$ instead of $y$. We assume that the exact solution is
\begin{align*}
x^{\dag}(t) =
\left\{
    \begin{array}{ll}
         0.5, \qquad & t\in[0.292,0.300], \\
         1, &      t\in[0.500,0.508], \\
         0.7, & t\in [0.700, 0.708],\\
         0, & \mbox{elsewhere}
             \end{array}
           \right.
\end{align*}
Let $y = Ax^{\dag}$ which is the exact data. For a given noise level $\d>0$, we add random Gaussian noise to $y$
to obtain $y^\d$ satisfying $\|y-y^\d\|_{L^2[0,1]} = \d$
which is used to reconstruct $x^\dag$ when the iteration is terminated by the
discrepancy principle (\ref{DP}).

In our numerical simulations, we take $x_0=0$ and $\xi_0=0$, we divide $[0,1]$ into $N=400$ subintervals of equal length,
approximate any integrals by the trapezoidal rule, and solve the involved minimization problems by the modified Fletcher-Reeves
CG method in \cite{ZZL2006}. In Figure \ref{linear1} we present the reconstruction results by taking $\d = 0.5\times 10^{-3}$
and $\alpha_n = 2^{-n}$ with $\tau=1.02$ in the discrepancy principle (\ref{DP}).
Figure \ref{linear1}(a) reports the result via the method with $\Theta(x) = \|x\|_{L^2}^2$. It is
clear that the reconstructed solution is rather oscillatory and fails to capture the sparsity of the exact solution $x^{\dag}$.
Figure \ref{linear1}(b) gives the result of the method with $\Theta(x) = \mu\|x\|_{L^2}^2+\|x\|_{L^1}$ and $\mu = 0.01$.
During the computation, $\|x\|_{L^1}$ is replaced by a smooth one $\int_0^1\sqrt{|x|^2+\epsilon}$ with $\epsilon = 10^{-6}$.
The sparsity reconstruction is significantly improved.

\end{example}

\begin{figure}[htp]
  \begin{center}
    \includegraphics[width = 0.9\textwidth, height= 2.8in]{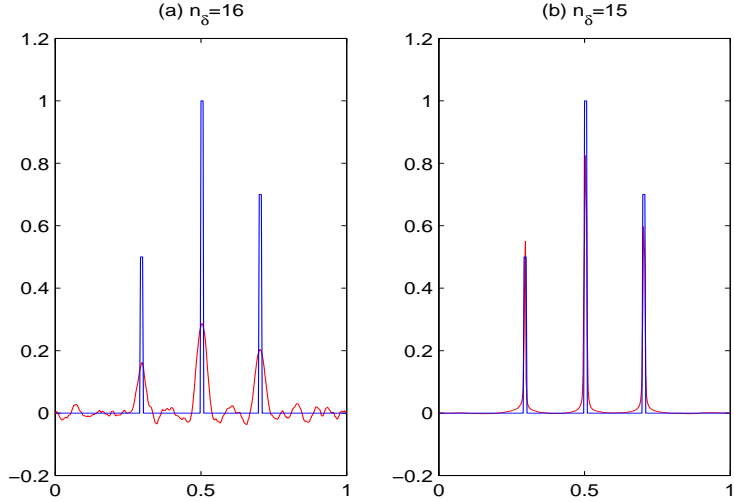}
  \end{center}
  \caption{\rm Reconstruction results for Example \ref{ex1}:
  $(a)$ $\Theta(x) = \|x\|_{L^2}^2$; $(b)$ $\Theta(x) = \mu \|x\|_{L^2}^2+\|x\|_{L^1}$
  with $\mu=0.01$.}
  \label{linear1}
\end{figure}

\begin{example}\label{ex4}
We next consider the identification of the parameter $c$ in the boundary value problem
\begin{align}\label{PDE}
\left\{
  \begin{array}{ll}
    -\triangle u + cu = f \qquad  \mbox{in } \Omega, \\
    u = g  \qquad \mbox{on } \partial\Omega
  \end{array}
\right.
\end{align}
from an $L^2(\Omega)$-measurement of the state $u$, where $\Omega\subset\mathbb{R}^d$, $d\leq 3$,
is a bounded domain with Lipschitz boundary, $f\in L^2(\Omega)$ and $g\in H^{3/2}(\p\Omega)$.
We assume that the sought solution $c^{\dag}$ is in $L^2(\Omega)$. This problem reduces to solving
an equation of the form (\ref{1.1}) if we define the nonlinear operator $F: L^2(\Omega)\rightarrow  L^2(\Omega)$ by
$F(c): = u(c)$, where $u(c)\in H^2(\Omega)\subset  L^2(\Omega)$ denotes the unique solution of (\ref{PDE}).
This operator $F$ is well defined on
\begin{align*}
D(F): = \left\{ c\in  L^2(\Omega) : \|c-\hat{c}\|_{ L^2(\Omega)}\leq \gamma_0 \mbox{ for some }
\hat{c}\ge 0 \mbox{ a.e.}\right\}
\end{align*}
for some positive constant $\gamma_0>0$. It is well known that $F$ is Fr\'{e}chet differentiable;
the Fr\'{e}chet derivative of $F$ and its adjoint are given by
\begin{align*}
F'(c)h  = -A(c)^{-1}(hF(c)) \quad \mbox{and} \quad F'(c)^* w  = -u(c) A(c)^{-1}w
\end{align*}
for $h, w\in L^2(\Omega)$, where $A(c): H^2\cap H_0^1\rightarrow  L^2$ is defined by $A(c) u = -\triangle u +cu$
which is an isomorphism uniformly in a ball $B_{\rho}(c_0) \cap D(F)$ for any $c_0\in D(F)$ with small $\rho>0$.
It has been shown (see \cite{EHN96}) that for any $\bar c, c\in B_{\rho}(c_0)$ there holds
\begin{align*}
\|F(\bar c)-F(c)- F'(c)(\bar c-c)\|_{L^2(\Omega)}\leq C\|\bar c-c\|_{L^2(\Omega)}\|F(\bar c)-F(c)\|_{L^2(\Omega)}.
\end{align*}
Therefore, Assumption \ref{A1} and the condition (\ref{A1-s}) hold if $\rho>0$ is small enough.

{\rm
In our numerical simulation, we consider the two dimensional problem with $\Omega = [0,1]\times[0,1]$ and
\begin{align*}
c^{\dag}(x,y) = \left\{
  \begin{array}{ll}
    1, \qquad & \hbox{if $(x-0.3)^2+(y-0.7)^2\leq 0.2^2$;} \\
    0.5, & \hbox{if $(x,y)\in [0.6,0.8]\times [0.2,0.5]$;} \\
    0, & \hbox{elsewhere.}
  \end{array}
\right.
\end{align*}
We assume $u(c^{\dag}) = x+y$ and add noise to produce the noisy data $u^{\delta}$ satisfying
$\|u^{\delta} - u(c^{\dag})\|_{L^2(\Omega)} = \delta$. We take $\delta = 0.1\times 10^{-3}$ and $\alpha_n = 2^{-n}$.
The partial differential equations involved are solved approximately by a finite difference method by dividing $\Omega$
into $40\times 40$ small squares of equal size. and the involved minimization problems are solved by the
modified nonlinear CG method in \cite{ZZL2006}. we take the initial guess $c_0=0$ and $\xi_0=0$, and terminate
the iteration by the discrepancy principle \eqref{DP} with $\tau = 1.05$.

\begin{figure}[htp]
  \begin{center}
  \includegraphics[width = 1\textwidth, height= 4.4in]{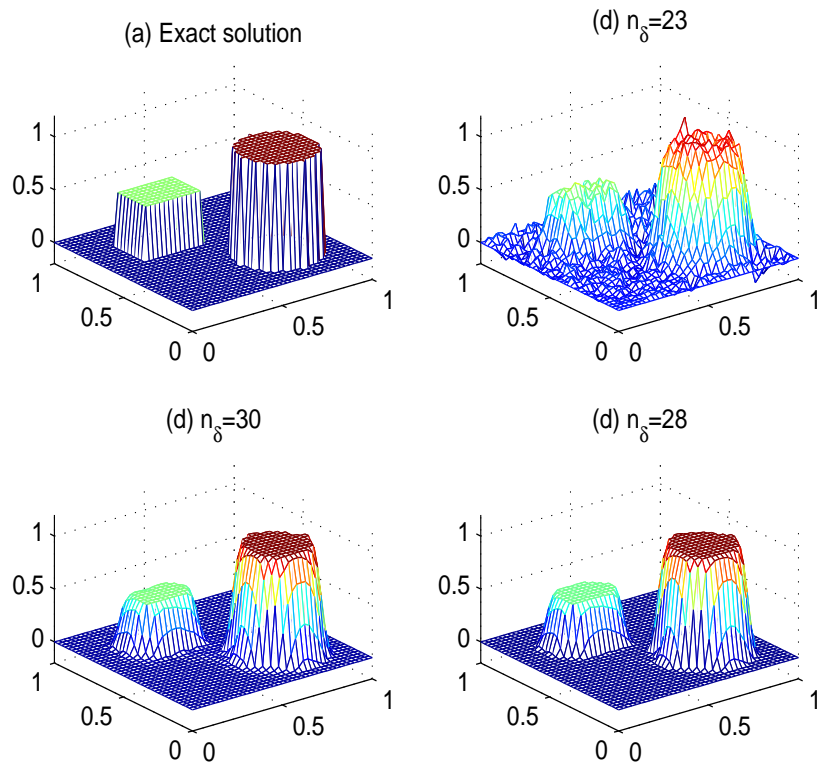}
  \end{center}
  \caption{\rm Reconstruction results for Example~\ref{ex4}: (a) Exact solution; (b) $\Theta(c) = \|c\|_{L^2}^2$;
  (c) and (d) $\Theta(c) = \mu\|c\|_{L^2}^2+\int_\Omega|Dc|$ with $\mu=0.01$ and $\mu=1$ respectively.}
  \label{nonlinear2}
\end{figure}

Figure \ref{nonlinear2}(a) plots the exact solution $c^{\dag}(x,y)$.  Figure \ref{nonlinear2}(b)
shows the result for the method with $\Theta(c) = \|c\|^2_{L^2}$.  Figure \ref{nonlinear2} (c) and (d)
report the reconstruction results for the method with $\Theta(c) = \mu\|c\|_{L^2}^2+\int_\Omega |Dc|$
for $\mu = 0.01$ and $\mu=1.0$ respectively; the term $\int_\Omega |Dc|$ is
replaced by a smooth one $\int_{\Omega}\sqrt{|Dc|^2+\epsilon}$ with $\epsilon = 10^{-6}$ during computation.
The reconstruction results in (c) and (d) significantly improve the one in (b) by efficiently removing the
notorious oscillatory effect and indicate that the method is robust with respect to $\mu$. We remark that, due to
the smaller value of $\mu$, the reconstruction result in (d) is slightly better than the one in (c) as can be seen
from the plots; the computational time for (d), however, is longer.
}
\end{example}

\vskip 0.3cm
\noindent
{\bf Acknowledgements}  Q Jin is partly supported by the grant DE120101707 of Australian Research Council,
and M Zhong is partly supported by the National Natural Science Foundation of China (No.11101093).

\end{document}